\RequirePackage[l2tabu, orthodox]{nag} 
\documentclass[oneside,english]{article}
\usepackage{dsfont}

\RequirePackage{amsthm,amsmath,amsfonts,amssymb}

\usepackage{mathptmx} 
\usepackage[T1]{fontenc} 
\usepackage{mathtools}
\usepackage[inline]{enumitem} 
\usepackage{babel} 
\usepackage[a4paper]{geometry}
\usepackage{graphicx} 
\usepackage{microtype} 
\usepackage{siunitx}
\usepackage{booktabs}
\usepackage[colorlinks,citecolor=blue,urlcolor=blue]{hyperref}
\usepackage{verbatim}
\usepackage{soul} 
\usepackage{longtable} 
\usepackage{subcaption}
\usepackage{tikz}
\usetikzlibrary{matrix,calc,positioning,decorations.pathreplacing}

\usepackage{algorithm}
\usepackage[noend]{algpseudocode}
\usepackage{bbm}

\newcommand{\Implies}[2]{$\text{\ref{#1}}\implies\text{\ref{#2}}$}
\newcommand{\Iff}[2]{$\text{\ref{#1}}\iff\text{\ref{#2}}$}

\newcommand{\EE}{\mathbb{E}} \newcommand{\PP}{\mathbb{P}}
 \newcommand{\id}{\mathrm{id}}

\newcommand{\Ban}{V}
\newcommand{\R}{\mathbb{R}} \newcommand{\N}{\mathbb{N}}

\newcommand{\pr}{\operatorname{Prob}}

\newcommand{\APr}{\mathcal{S}} 
\newcommand{\bk}{\mathbf{k}} 
\newcommand{\bX}{\mathbf{X}} 
\newcommand{\bZ}{\mathbf{Z}}

 \newcommand{\bY}{\mathbf{Y}}
 \newcommand{\cP}{\mathcal{P}}
\newcommand{\Prob}{\mathbb{P}} \newcommand{\cp}{\mathrm{AF}}

\newcommand{\cM}{\mathcal{M}} \newcommand{\cK}{\mathrm{K}}
\newcommand{\cF}{\mathcal{F}} \newcommand{\cG}{\mathcal{G}}
 
\newcommand{\cH}{\mathcal{H}}

\newcommand{\Sig}[1]{\operatorname{S}({#1})}

\newcommand{\Sigo}{\operatorname{S}}
\newcommand{\sig}[2]{\operatorname{S}_{#1}({#2})}
\newcommand{\sigr}[1]{\operatorname{S}_{#1}}
\newcommand{\esigr}[1]{\bar\Sigo_{#1}}

\newcommand{\sigrn}[2]{\operatorname{S}_{#1}^{n}({#2})}
\newcommand{\sigrnhat}[2]{\tilde{\operatorname{S}}_{#1}^{n}({#2})}
\newcommand{\esigrn}[2]{\bar\Sigo_{#1}^{n}(#2)}
\newcommand{\sigrno}[1]{\operatorname{S}_{#1}^{n}}
\newcommand{\esigrno}[1]{\bar\Sigo_{#1}^{n}}

\newcommand{\Seq}{\operatorname{Seq}}

\newcommand{\ESig}[2]{\bar \Sigo_{#1}(#2)}
\newcommand{\esig}[1]{\bar\Sigo_{#1}}
\newcommand{\esigo}{\bar \Sigo}

\newcommand{\idx}{I}  \newcommand{\vs}{V}
 \newcommand{\tops}{\mathrm{U}}

\newcommand{\rank}[1]{\operatorname{rank}(#1)}
 
\newcommand{\Meas}[1]{\operatorname{Meas}(#1)}
\newcommand{\MeasA}[1]{\operatorname{Meas}_a(#1)}

\newcommand{\meas}[1]{\operatorname{Meas}_{#1}}

\newcommand{\prob}[1]{\operatorname{Prob}(#1)}

\newcommand{\Law}[1]{\operatorname{Law}(#1)}

\newcommand{\cX}{\mathcal{X}} 

\newcommand{\ta}[2]{{\operatorname{\mathbf{ut}}^{#1}({#2})}}
\newcommand{\ata}[2]{{\operatorname{\mathbf{t}}^{#1}({#2})}}
\newcommand{\Ta}[2]{{\operatorname{\mathbf{T}}^{#1}({#2})}}
\newcommand{\TaRank}[1]{{\operatorname{\mathbf{T}}^{#1}}}

\newtheorem{theorem}{Theorem} 
\newtheorem{corollary}{Corollary}

\newtheorem{proposition}{Proposition} 
\newtheorem{prop}[theorem]{Proposition}

\newtheorem{definition}{Definition}

\newtheorem{remark}{Remark} 
\newtheorem{lemma}{Lemma}
\newtheorem{example}{Example}[section] 
\newtheorem*{ackn}{Acknowledgements}

\begin{document}

\title{Adapted Topologies and Higher Rank Signatures}
\author{
	Patric Bonnier$ ^{\ast} $, Chong Liu$ ^{\ast} $, Harald Oberhauser$ ^{\ast} $ \\[5pt]
	\null$^\ast$ Mathematical Institute, University of Oxford \\
	{\small\texttt{Patric.Bonnier@\hspace{0.1pt}maths.ox.ac.uk}, \quad
	\small\texttt{Chong.Liu@\hspace{0.1pt}maths.ox.ac.uk}, \quad 
	\texttt{Harald.Oberhauser@\hspace{0.1pt}maths.ox.ac.uk}}
}

\date{}
\maketitle

\begin{abstract}
  The topology of weak convergence does not account for the growth of information over time that is captured in the filtration of an adapted stochastic process.
  For example, two adapted stochastic processes can have very similar laws but give completely different results in applications such as optimal stopping, queuing theory, or stochastic programming. 
  To address such discontinuities, Aldous introduced the extended weak topology, and subsequently, Hoover and Keisler showed that both, weak topology and extended weak topology, are just the first two topologies in a sequence of topologies that get increasingly finer.
  We use higher rank expected signatures to embed adapted processes into graded linear spaces and show that these embeddings induce the adapted topologies of Hoover--Keisler.
\end{abstract}

\section{Introduction}
A sequence of $\R^d$-valued random variables $(X_n)_{n\geq0}$ is said to converge weakly to a random variable $X$ if 
\begin{align}\label{eq: weak conv}
 \int_{\R^d} f(x) \Prob_n(X_n \in dx)  \rightarrow \int_{\R^d} f(x) \Prob(X \in dx)\quad \text{ for all } f \in C_b(\R^d,\R). 
\end{align}
If one replaces $\R^d$-valued random variables by path-valued random variables -- that is random maps from a totally ordered set $\idx$\label{idx} into $\R^d$ -- one arrives at the definition of weak convergence of stochastic processes. 
However, reducing stochastic processes to path-valued random variables ignores the filtration of the process.
Filtrations encode how the information one has observed in the past restricts the future possibilities and thereby encodes actionable information.
Even for discrete-time and real-valued Markov processes equipped with their natural filtrations, the weak topology is sometimes too coarse, see Example~\ref{ex:ost}.
\begin{example}[\cite{Aldous81,Beiglbock19}]\label{ex:ost}\quad
  \begin{enumerate}
  \item \label{itm:ex1}
	The value map of an optimal stopping problem,
	\begin{align} \label{eq:ost}
    \bX\coloneqq (\Omega, (\cF_t)_{t \in \idx}, \mathbb{P}, (X_t)_{t \in \idx}) \mapsto \inf_\tau\EE[L_\tau],
	\end{align}
  where the $\inf$ is taken over all stopping times $\tau \le T$ is not continuous in the weak topology if $L$ is an adapted functional that depends continuously on the sample path of $\bX$.
  This discontinuity remains even if the domain of the solution map \eqref{eq:ost} is restricted to the space of discrete-time Markov processes equipped with their natural filtration. 
\item \label{itm:ex2}
  Figure~\ref{fig:example} shows a sequence of Markov processes that converge weakly.
  However, at time $t=1$ one would make very different decisions upon observing the process for finite $n$ and its weak limit (e.g.~for portfolio allocations of investments or in optimal stopping problems such as \eqref{eq:ost}).
  The reason for this discontinuity is that although the law of the processes gets arbitrarily close for large $n$, their natural filtrations are very different. 
  \end{enumerate}
\end{example}
\begin{figure}[t]
	\begin{tikzpicture}[shorten >=1pt,draw=black!50]
	\draw[black, fill=black, thick] (0,0) circle (2pt);
	\draw[black, fill=black, thick] (3,.2) circle (2pt);
	\draw[black, fill=black, thick] (3,-.2) circle (2pt);
	\draw[black, thick] (0,0) -- (3,.2);
	\node () at (1.5,.4) {\footnotesize$ p=0.5 $};
	\draw[black, thick] (0,0) -- (3,-.2);
	\node () at (1.5,-.5) {\footnotesize$ p=0.5 $};
	\draw [black, thick, decorate,decoration={brace,amplitude=4pt}] (3.15,0.2) -- (3.15,-.2) node [black,midway,xshift=9pt] {\footnotesize$\frac{1}{n}$};
	\draw[black, thick] (3,.2) -- (6,1);
	\node () at (4.5,.9) {\footnotesize$ p=1 $};
	\draw[black, thick] (3,-.2) -- (6,-1);
	\node () at (4.5,-1) {\footnotesize$ p=1 $};
	\draw[black, fill=black, thick] (6,1) circle (2pt);
	\draw[black, fill=black, thick] (6,-1) circle (2pt);
	
	\draw[black, fill=black, thick] (7,0) circle (2pt);
	\draw[black, fill=black, thick] (10,0) circle (2pt);
	\draw[black, thick] (7,0) -- (10,0);
	\node () at (8.5,.2) {\footnotesize$ p=1 $};
	\draw[black, thick] (10,0) -- (13,1);
	\node () at (11.5,.8) {\footnotesize$ p=0.5 $};
	\draw[black, thick] (10,0) -- (13,-1);
	\node () at (11.5,-.9) {\footnotesize$ p=0.5 $};
	\draw[black, fill=black, thick] (13,1) circle (2pt);
	\draw[black, fill=black, thick] (13,-1) circle (2pt);
	\end{tikzpicture}
	\caption{A typical example of when weak convergence is insufficient. The process on the left can be made arbitrarily close to the process on the right as $ n \to \infty $.}
	\label{fig:example}
\end{figure}
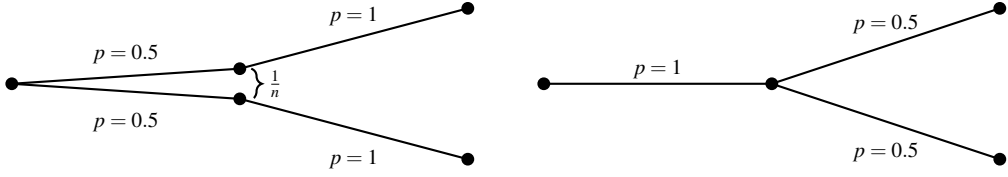

\subsection{Adapted Topologies}
Such shortcomings of weak convergence for stochastic processes were recognized and addressed in the 1970's and 1980's. 
Denote by $\APr$\label{apr} the set of adapted processes that evolve in a state space that is compact subset of $\R^d$. 
David Aldous proposed to associate with an adapted process $\bX = (\Omega, (\cF_t)_{t \in \idx}, \mathbb{P}, (X_t)_{t \in \idx}) \in \APr$ its \emph{prediction process} $\hat \bX= (\Omega, (\cF_t)_{t \in \idx}, \mathbb{P}, (\hat X_t)_{t \in \idx})$, 
\begin{align}
\hat X_t \coloneqq \Prob(X \in \cdot | \cF_t),
\end{align}
and to define a topology on $\APr$ by prescribing that two processes converge if and only if their prediction processes converge in the weak topology (that is, weak convergence in the space of measure-valued processes).
Aldous studied this topology in \cite{Aldous81} and showed that it has several attractive properties such as making the map in Example \ref{ex:ost} item \ref{itm:ex1} continuous and separating the two processes in item \ref{itm:ex2}.
Similar points were also made and further developed by a number of different researchers \cite{Vershik70, Vershik94, Lasalle18, Ruschendorf85, Veraguas_2020, eder2019compactness,Wiesel20} including ones in other communities such as economics \cite{Hellwig96}, operations research \cite{Pflug12, Pichler13, Pflug14, Pflug15, Pflug16}, and numerics \cite{Talay19} and has led to the development of topologies that are finer than the classical weak topology.
The construction of all these differ in detail, but in discrete time and under the natural filtration they lead to the same topology that Aldous originally introduced as was recently shown in \cite{Beiglbock19}.
We henceforth refer to this topology\footnote{Aldous refers to this topology as the \emph{extended weak topology}.} as the \emph{adapted topology of rank $1$} and we refer to the classic weak topology as the \emph{adapted topology of rank $0$} (denoted $ \tau_1 $ and $ \tau_0 $ respectively).

However, even the adapted topology of rank $1$ (weak convergence of the prediction process) does not characterize the full structure of adapted processes, as evidenced by Example \ref{ex:HK}
\begin{example}[Example 3.2, \cite{Hoover84}]\label{ex:HK}
  There exists two sequences of Markov chains, $(\bX_n)_n $ and $ (\bY_n)_n $, that both converge to the same process in the topology $\tau_0$ and in the topology $\tau_1$ as $n \to \infty$.
  However, the information contained in their filtrations is still different; for example $ \EE[\EE[\bX^n_4\vert \cF_3]^2\vert \cF_1]-\EE[\EE[\bY^n_4\vert \cF_3]^2\vert \cF_1] \not\to 0 $ as $n \to \infty$; see Appendix \ref{app:A} for details. 
\end{example}
Seminal work of Hoover--Keisler \cite{Hoover84} provides a definite answer: it shows the existence of a sequence of topologies $(\tau_r)_{r \ge0}$ on $\APr$ that become strictly finer as $r$ increases; $\tau_0$ is the topology of weak convergence; $\tau_1$ is Aldous' weak convergence of prediction processes, and $\bigcap_{r=0}^{\infty} \tau_r $ identifies two process if and only if they are isomorphic, see \cite{Hoover84} for the precise statement. 
We refer to $\tau_r$ as the \emph{adapted topology of rank $r$} on $\APr$.
The approach of \cite{Hoover84} is different than Aldous' approach that relies on prediction processes.
The starting point of \cite{Hoover84} is that one may specify a topology by choosing a class of functionals on pathspace, that is a subset of $ (\R^d)^\idx \rightarrow \R$, and define convergence of a sequence of processes $\bX_n$ to $\bX$ by requiring
\begin{align}\label{eq: weak via cp}
  \bX_n \rightarrow \bX  \text{, if and only if } 
  \int_{(\R^d)^I} f(x) \Prob_n(X_n \in dx)  \rightarrow \int_{(\R^d)^I} f(x) \Prob(X \in dx), 
  \end{align}
  for all $f$ in this set of functionals.
  By taking this set of functionals to be $ C_b((\R^d)^\idx, \R)$ one recovers weak convergence, but much richer classes of functionals can be constructed by iterating conditional expectations and compositions with bounded continuous functions, e.g.~$f(X(\omega)) = \EE[\cos (X_{t_1}X_{t_2})|\cF_{t_3}](\omega)$ is one such function. 
  In fact, to avoid measure-theoretic trouble, it is more convenient to work with random variables:
  one defines so-called adapted functionals $\cp$\label{cp} as maps from $\APr$ to the space of real-valued random variables, by mapping a process $\bX$ to a random variable given as above by iteration of finite marginals, conditional expectations, and continuous functions.
  The minimal number of nested conditional expectations needed to specify an element of $\cp$ induces the natural grading
  \begin{align}
  \cp = \bigcup_{r \ge 0}\cp_{r}
  \end{align}
  where $\cp_r$\label{cpr} denotes all adapted functionals that are build with $r$ nested conditional expectations. 
Hoover--Keisler showed that by defining \[\bX_n \rightarrow \bX \text{ if and only if }\EE[f(\bX_n)] \rightarrow \EE[f(\bX)] \text{ for all }f \in \cp_r,\] then the associated topologies get strictly finer as $r \rightarrow \infty $; e.g. the topology $\tau_2$ separates the two adapted processes in Example \ref{ex:HK}.

\subsection{Contribution}
Denote by $\APr=\APr(\R^d)$ the set of adapted stochastic processes $
\bX=(\Omega, (\cF_t)_{t \in \idx}, \mathbb{P}, (X_t)_{t \in \idx})$ that evolve
in $\R^d$. The main contribution of this article is to provide for every $r \ge 0$ an explicit map
\begin{align}
  \Phi_r: \APr \to \TaRank{r+1},\quad \bX \mapsto \Phi_r(\bX) 
\end{align}
from $\APr$ into a normed and graded space $\TaRank{r+1}$ such that the adapted topology of rank $r$ on $\APr$, $\tau_{r}$, arises as the
\emph{initial topology} for $\Phi_{r}$. That is $\tau_{r}$ is the coarsest
topology $\tau$ on $\APr$ that makes the map \[\Phi_{r}:(\APr,\tau) \to
(\TaRank{r+1},\|\cdot \|)\] continuous. Equivalently, the adapted topology of
rank $r$, $\tau_r$, is characterized by the universal property that any map $f$
from a topological space into $\APr$ is continuous if and only if $\Phi_{r}
\circ f$ is continuous. We highlight three consequences of this result:
\begin{description}
\item[Metrizing adapted topologies of any rank $r$.] It immediately follows that
  \begin{align}\label{eq:semimetric}
    \APr\times \APr \to [0,\infty),\quad  (\bX, \bY) \mapsto \| \Phi_{r}(\bX) - \Phi_{r}(\bY)\|
  \end{align}
  is a semi-metric on $\APr$ that induces $\tau_r$. For $r=0$ and general
  stochastic processes our results reduce to the previously known result
  \cite{Oberhauser18} that the expected signature map $\Phi_0$ can metrize weak
  convergence; for $r=1$ this adds a novel entry to the list of semi-metrics
  that induce $\tau_1$, see \cite{Beiglbock19}; for $r\ge 2$ this (semi-)metric
  seems to be the first metrization\footnote{However, we draw attention to
    forthcoming work of G.~Pammer et al.} of $(\APr,\tau_{r})$.
  Further, our results are not restricted to processes equipped with their natural filtration.

\item[Dynamic Programming.] For $r=0$, the map $\Phi_0$ reduces to the expected
  signature map. A direct application of dynamic programming shows that for a
  Markov process $\bX$, $\Phi_0(\bX)$ and consequently the semi-metric
  \eqref{eq:semimetric}, can be efficiently computed by dynamic programming. For
  $r \ge 1$, the maps $\Phi_r$ are constructed by recursion and we show this can
  be used to bootstrap dynamic programming principles, so that for any $r \ge 0$
  the map $\Phi_r$ resp.~the semi-metric \eqref{eq:semimetric} can be
  efficiently computed for Markov processes.
\item[A multi-graded ``feature map''.] The maps $\Phi_r$ embed a stochastic
  process into linear spaces $\TaRank{r+1}$ that arise via a classic free
  construction in algebra, namely \emph{the free algebra functor}. In
  particular, $\TaRank{r+1}$ has a natural multi-grading and $\Phi_r(\bX)$ use
  this to describe the interplay of the law and the filtration of the process
  $\bX$ in a hierarchical manner; analogous to how the classical moments of a
  vector-valued random variable is graded by the moment degree.
\end{description}
We believe the last point is the strongest contribution of this approach to the
existing literature since the embedding \[\bX \to \Phi_r(\bX)\] of an adapted
process $\bX$ into a multi-graded linear space $(\TaRank{r+1}, \| \cdot \|)$
delivers more than a semi-metric.
This seems to be novel even for the well-studied case of $r=1$.
For example, for $r=0$, $\Phi_0$ is just the expected signature map and many recent applications in statistics, machine learning and finance rely on the co-ordinates and the grading of $\Phi_0(\bX)$.
In Section~\ref{sec: experiments} we give a simple supervised classification example that demonstrates how expected signatures as they are currently used in machine learning, i.e.~$\Phi_0$, can yield a too coarse description even for simple Markov processes and how this is resolved by $\Phi_r$ for $r \ge 1$.
We also mention that adapted topologies (so far, via causal Wasserstein semi-metric) are finding applications in machine learning, see \cite{xu2020cot}, and the use of $\Phi_r$ in this context seems to be interesting future research venue.

\begin{remark}
  We focus on finite discrete time processes for two reasons:
  \begin{enumerate*}[label=(\roman*)]
  \item Most applications and in fact, much of the recent literature on adapted
    topologies, studies finite discrete time.

  \item The resulting signature and tensor structure that capture filtrations
    are already novel and interesting to study in finite discrete time. Some
    definitions and results immediately extend to continuous time, but others
    lead quickly to challenging research programmes;
    e.g.~for $r=1$ the prediction process $ t \mapsto \hat \bX^1_t = \PP(X \in
    \cdot|\cF_t)$ has only c\`adl\`ag trajectories, even if the sample paths of
    $t \mapsto X_t$ are continuous. C\`adl\`ag rough path theory is an area of
    ongoing research~\cite{Chevyrev19,friz2017} and the question of how
    tightness propagates through such iterated (higher rank) constructions seems
    hard due to a lack of Prohorov type results; see Section \ref{sec:tightness}
    for details.
  \end{enumerate*}
\end{remark}
\begin{remark}
  Our results are not restricted to stochastic processes evolving in compact subsets of finite-dimensional state spaces discussed above.
  By using robust signatures~\cite{Oberhauser18} adapted processes that evolve in general separable Banach space are included in our approach.
  In this non-compact case, it turns out that the Hoover--Keisler approach of specifying an adapted topology via $\cp_r$ and the natural generalization of Aldous's approach given by iterating prediction process yield in general different topologies which might of independent interest.
\end{remark}

\subsection{Outline and Notation.}
The rest of the paper is laid out as follows:
\begin{itemize}
	\item Section \ref{sec: adapted topology} recalls Hoover--Keisler's \emph{adapted functionals} $\cp= \bigcup_{r \ge 0} \cp_r$ and the adapted topology of rank $r$, $\tau_r$.
    Further, it identifies Aldous prediction $\hat X^1$ as the rank $r=1$ construction in the sequence of rank $r$ prediction process that we define as \[\hat X^{r+1}_t \coloneqq \Prob(\hat X^{r} \in \cdot | \cF_t),\quad\hat X^0 \coloneqq X.\]
    These prediction processes evolve in state spaces that have a rich structure; e.g. 
    \begin{align}\label{eq:metric}
      \Law{\hat \bX^0} &\in \Meas{ \idx \rightarrow \vs}\eqqcolon \cM_1,\\
      \Law{\hat \bX^1}&\in \Meas{ I \rightarrow \Meas{\idx \rightarrow \vs}} \eqqcolon \cM_2, \\
      \Law{\hat \bX^3} &\in \Meas{ I \rightarrow \Meas{ \idx \rightarrow \Meas{\idx \rightarrow \vs}}} \eqqcolon \cM_3.  
    \end{align}
    We refer to the spaces $\cM_r$ as \emph{rank $r$ measures}.
    Capturing their structure is the central theme of this article.
	\item Section \ref{sec: sig} discusses how an element of $\cM_r$ can be described by a multi-graded sequence of tensors.
    For $r=0$, we recall that the signature $\sigr{1}$ injects a path into the free algebra $\TaRank{1}$ that consists of sequences of tensors of increasing degree; the expected signature $\esig{1}$ injects $\cM_1$ into $\TaRank{1}$.
    To generalize this from $r=1$ to general $r\ge 1$ we first introduce the space of \emph{higher rank paths} $\vs_r$: for a linear space $\vs$ define
    \[
      \vs_{r+1} \coloneqq I \to \vs_r\quad \vs_0\coloneqq \vs.
      \]
    The \emph{rank $r$ signature} $\sigr{r}:\vs_{r} \to \TaRank{r}$ then injects a rank $r$ path into the \emph{rank $r$ tensor algebra} $\TaRank{r}$ which consists of sequences of multi-graded sequences of tensors; the \emph{rank $r$ expected signature} $\esig{r}: \cM_r \to \TaRank{r}$ provides a multi-graded description of an element of $\cM_r$ by injecting it into $\TaRank{r}$.  

	\item Section \ref{sec: top and sig} contains our main theoretical results.
    We first show that convergence in the adapted topology $\tau_r$ is equivalent to convergence in law of the rank $r$ prediction process.
    This allows us to show that the rank $r+1$ expected signature $\esig{r+1}$ applied to the rank $r$ prediction process induces the rank $r$ topology $\tau_r$.
    Hence, the map 
    \begin{align}
      \Phi_r: \APr \to \TaRank{r+1},\quad \Phi_r(\bX) \coloneqq \esig{r+1}(\Law{\hat X^r})
    \end{align}
    induces $\tau_{r}$ as initial topology on $\APr$.
  \item Section \ref{sec: experiments} shows that the maps $\Phi_r(\bX)$ can be efficiently computed by dynamic programming when $\bX$ is a Markov process.
    We provide a  Python implementation\footnote{Available at \url{https://github.com/PatricBonnier/Higher-rank-signature-regression}} of the resulting algorithms and use it for a simple numerical experiment that demonstrates the advantages of $\Phi_1$ against the usual expected signature $\Phi_0$. 
  \item Appendix \ref{app:A} contains details for Example \ref{ex:HK}, Appendix \ref{app:tensor algebras} contains some details on the construction of higher rank tensor algebras, and Appendix \ref{app:normalization} contains some background on the robust signature and how it can be used to overcome problems arising from non-compactness.
\end{itemize}

  \begin{center}

    \renewcommand{\arraystretch}{1.1}
    \begin{longtable}{llr}\toprule
      Symbol & Meaning & Page\\
      \toprule \multicolumn{2}{c}{Spaces}
      \\
      \midrule
      $\Ban$ & a separable Banach space & \pageref{bs}\\
      $\tops$ & a topological space & \pageref{tops}\\
      $\APr(\tops)$ & the set of adapted stochastic processes in $ \tops $ & \pageref{apr} \\
      $\underline \Omega$& an adapted probability space $\underline \Omega = (\Omega,\mathbb{P}, (\cF_t))$ & \pageref{omega} \\
      $\bX=(\underline \Omega, X) \in \APr(\tops)$& an adapted process on the stochastic base $\underline \Omega$ & \pageref{bX}\\
      $\Meas{\tops}$ & Borel measures on $\tops$ & \pageref{Meas} \\
      $\pr(\tops)$ & Borel probability measures on $\tops$ & \pageref{pr} \\
      $ \idx $ & A finite totally ordered set (time) & \pageref{idx} \\
      $ \big( I \to \tops\big) $ & the space of sequences in $ \tops $ indexed by $ I $ & \pageref{paths} \\
      \midrule
      \multicolumn{2}{c}{The Adapted Topology of Rank $r$} & \null \\
      \midrule
      $\cp$ & adapted functionals, $f(\bX)$ is a real-valued random variable & \pageref{cp} \\ 
      $\cp_r=\{f \in \cp\,|\, \rank{f}\leq r\}$ & adapted functionals with rank less than $r$ & \pageref{cpr} \\ 
      $\tau_r$ & the adapted topology of rank $r$ on $\APr(\tops)$ & \pageref{adapted topology} \\
      $\hat \tau_r$ & the extended weak topology of rank $r$ on $\APr(\tops)$ & \pageref{extended weak topology}\\
      \midrule
      \multicolumn{2}{c}{Paths and Measures of Rank $r$} & \null \\
      \midrule
      $ \tops_r $ & the space of rank $ r $ paths with state space $ \tops $ & \pageref{topsr} \\
      $ \cM_r(\tops) $ & the space of rank $ r $ Borel measures on $ \tops $ & \pageref{cmr} \\
      $ \cP_r(\tops) $ & the space of rank $ r $ Borel probability measures on $ \tops $ & \pageref{prr} \\
      \midrule
      \multicolumn{2}{c}{(Expected) Signature of Rank $r$} & \null \\
      \midrule
      $ \Ta{r}{\Ban}$ & The rank $ r $ tensor algebra &\pageref{rtensor} \\
      $\hat \bX^r$ & the rank $r$-prediction process $\hat \bX_t^r:= \mathbb{P}(\hat \bX^{r-1} \in \cdot|\cF_t)$ of $\bX \in \APr$ & \pageref{bXr} \\
      $\sigr{r}$ & the rank $r$ signature map $\sigr{r}: \Ban_r \to \Ta{r}{\Ban}$ & \pageref{sigr} \\
      $\esig{r}$ & the rank $r$ expected signature map $\esig{r}: \cM_r(\Ban) \to \Ta{r}{\Ban}$& \pageref{esigr} \\
      $\bar \bX^r$ & the rank $r$ conditional expected signature $\bar \bX^r_t:=\EE[\bar \bX^{r-1}|\cF_t]$ & \pageref{barbxr} \\
      $d_r(\bX,\bY)$ & the rank $r$ adapted signature distance between $\bX$ and $\bY$ & \pageref{distance}\\
      \bottomrule
    \end{longtable}
  \end{center}

\section{The Adapted Topology $\tau_r$ and the Extended Weak Topology $\hat \tau_r$}\label{sec: adapted topology}
In this section we recall work of Hoover--Keisler \cite{Hoover84}, and define adapted functionals $\cp_r$ of rank $r$.
We then revisit Aldous \cite{Aldous81} notion of a prediction process, and generalize it to rank $r$ prediction processes; that is we associate with every element $\bX \in \APr(\tops)$ the sequence $(\hat \bX^r)_{r \ge 0}$ of rank $r$ prediction processes.
Both of the resulting objects -- adapted functionals of rank $r$ resp.~prediction processes of rank $r$ -- can be used to define a topology on the space $\APr(\tops)$ of adapted processes with state space $\tops$ and intuitively capture more structural information of the filtration as $r$ increases.
We refer to these two topologies as the adapted topology $\tau_r$ of rank $r$ and the extended weak topology of rank $r$. 

\begin{definition}

  Denote by $\idx=\{0,1,\ldots,T\}$. 
  A filtered probability space is a triple $\underline{\Omega}=(\Omega,\mathbb{P}, (\cF)_{t\in \idx})$\label{omega} consisting of a sample space $\Omega$, a probability measure $\mathbb{P}$, and a filtration $(\cF_t)_{t \in \idx}$.
  An adapted stochastic process $\bX = (\underline \Omega, X)$ with state space $\tops$\label{tops} consists of a filtered probability space $\underline \Omega$ and a map $X:\Omega \times \idx \rightarrow \tops$ such that $X_t$ is $\cF_t$-measurable for each $t \in \idx$.
  Denote with $\APr(\tops)$ the space of adapted stochastic processes that evolve in discrete time $\idx$ in a state space $\tops$, 
  \begin{align}
   \APr(\tops) \coloneqq \{ \bX \,|\, \bX \text{ is an adapted stochastic process indexed by $\idx$ that evolves in } \tops \}. 
  \end{align}
  We also set \[\Law{\bX}\coloneqq\Prob \circ X^{-1} \text{ for }\bX=(\Omega,\mathbb{P}, (\cF)_{t\in \idx},X).\]
  With the usual slight abuse of notation we use throughout the same symbol $\EE$ for the expectation although the elements of $\APr(\tops)$ can be supported on different adapted probability spaces.
\end{definition}
\subsection{Adapted Functionals}
A natural way to define a topology on $\APr(\tops)$ is by specifying some set of functionals and requiring that 
\begin{align}
\bX_n \rightarrow \bX \text{ if }\EE[f(\bX_n)] \rightarrow \EE[f(\bX)]  \text{ as } n \to \infty
\end{align}
for every $f$ in this set of functionals.
By choosing the set of functionals to be \[\{f\,|\,f(\bX)= g(X_{t_1},\ldots, X_{t_n}),\, g \in C_b(\tops^n,\R),\, (t_1,\ldots,t_n) \in \idx^n\}\] one recovers classical weak convergence. 
In view of the above examples, it is natural to construct a wider class of functionals by using the conditional expectation in order to capture some of the information contained in the filtration. 
\begin{definition}\label{def: adapted functionals}
  We define a set of maps $\cp$ from $\APr(\tops)$ into the set of real-valued random variables inductively: 
  \begin{enumerate}
  \item 
    if $t_1,\ldots,t_n \in \idx$ and $f \in C_b(\tops^n,\R)$, then $\bX \mapsto f(X(t_1),\ldots,X(t_n)) \in \cp$,
  \item 
    if $f_1,\ldots,f_n \in \cp$ and $f \in C_b(\R^n,\R)$, then $\bX \mapsto  f(f_1(\bX),\ldots,f_n(\bX)) \in \cp $,
  \item
    if $f \in \cp$ and $t \in \idx$ then $\bX \mapsto \EE[f(\bX)|\cF_t] \in \cp$.
  \end{enumerate}
  We refer to the elements of $\cp$ as adapted functionals\footnote{In \cite{Hoover84} $\cp$ are called conditional processes.}.
\end{definition}
\begin{remark}
  For a given $\bX = (\Omega,\mathbb{P}, (\cF)_{t\in \idx},X)$ and $f \in \cp$, $f(\bX) $ is in $ L^\infty(\Omega, \PP)$, hence the image set of $f \in \cp$ is $\prod_{\bX \in \APr(\tops)} L^\infty (\Omega^{\bX}, \PP^{\bX})$ where
  we write $\bX = (\Omega^{\bX},\mathbb{P}^{\bX}, (\cF^{\bX})_{t\in \idx},X)$ to emphasize the dependence of the underlying filtered probability spaces on $\bX$.
\end{remark}
Intuitively, the more times the conditional expectation is iterated the more of the evolutional constraints that are encapsulated in the filtration are exposed by the functionals in $ \cp $. 
Indeed, Figure \ref{fig:example} shows two processes that can not be distinguished without at least one iteration, and in Example \ref{ex:HK}, at least two iterations are required.
With this in mind, we define the rank $r$ of an adapted functional $f \in \cp$ as the minimal number of times the conditional expectation is iterated in the construction of $f$.
This number $r$ of conditional expectations gives $ \cp $ a natural grading. 
\begin{definition}\label{def: rank of adapted functionals}
	Define $\operatorname{rank}: \cp \rightarrow \N \cup \{0\}$ as 
	\begin{enumerate}
	\item
	  $\rank{f}=0$ if $f(\bX)=g(X_{t_1},\ldots,X_{t_n})$ for $g \in C_b(\tops^n,\R)$ 
	\item
	  $\rank{f}=\max (\rank{f_1},\ldots,\rank{f_n})$ if $f(\bX)=g(f_1(\bX), \ldots,f_n(\bX))$, $ g \in C_b(\R^n,\R) $,  $f_1,\ldots,f_n \in \cp$, 
	\item
	 $\rank{f}=\rank{g}+1$ if $f(\bX)=\EE[g(\bX)\vert \cF_t]$ for $g \in \cp$.
	\end{enumerate}
	We call 
	\begin{align}
	 \cp_r:=\{ f \in \cp | \rank{f}\leq r\} 
	\end{align}
	the set of adapted functionals of rank less than $r$.
\end{definition}
\begin{remark}
Following Definition \ref{def: adapted functionals}, every $f \in \cp$ can be obtained by repeating steps 1, 2 and 3 finitely many times. Let $\mathfrak{r}_f$ denote such an iterative procedure which leads to the construction of $f \in \cp$, and  let $|\mathfrak{r}_f|$ denote the total number of times step 3 (taking conditional expectation) appears in $\mathfrak{r}_f$. Note that $f$ does not uniquely determine $\mathfrak{r}_f$; for instance, $f(\bX) = g(X_{t_1}, \ldots,X_{t_n}) = \EE[g(X_{t_1}, \ldots,X_{t_n})|\cF_T]$ holds for all $\bX \in \APr(\tops)$ if $g$ is a constant function. So, strictly speaking, the map $\rank{f}$ is given by $\rank{f} := \min\{|\mathfrak{r}_f|: \mathfrak{r}_f \text{ is a representation of f} \}$. However, the above (strictly speaking, not well-defined) Definition \ref{def: rank of adapted functionals} is more intuitive.
\end{remark}

\subsection{Prediction Processes of Rank $r$}
We now revisit Aldous' notion of prediction process.
By introducing ``prediction processes of prediction processes'' one arrives at another natural sequence of objects (prediction processes of rank $r$) that capture more structure of the filtration. 
\begin{definition}\label{def: higher rank prediction process}
	Let $\bX=(\underline \Omega, X) \in \APr(\tops)$ \label{bX}. The adapted stochastic processes $(\hat \bX^r)_{r \ge 0}$ of $\bX$ are defined as
	$\hat \bX^r=(\underline \Omega, \hat X^r)$ \label{bXr} with $\hat X^r$ given inductively as 
	\begin{align}
	&\hat X^0:= X \text{ and }\hat X^{r+1}:= (\mathbb{P}(\hat X^r \in \cdot|\cF_t))_{t \in \idx}.
	\end{align}
	We call $\hat \bX^r$ the rank $r$ prediction process of $\bX$ and we denote with $\tops_r$ the state space of the process $\hat X^r$.
\end{definition}
An immediate but useful identity that we use several times is that 
\begin{align}
  \hat X^r_0 = \Law{\hat X^{r-1}}.
\end{align}
\subsection{The Adapted and the Weak Extended Topology of Rank $r$}
We now have two natural ways to generalize the definition of weak convergence so that it takes the filtration into account: one by replacing continuous bounded functions by adapted functions; one by replacing weak convergence of the process by weak convergence of the prediction process.
\begin{definition}
  \label{def: adapted distribution and topology}
  \label{def: higher rank extended weak topology}
  Let $r\ge 0$. We say that two adapted processes $\bX \in \APr(\tops)$ and $\bY \in \APr(\tops)$ have the same adapted distribution up to rank $r$, in notation $\bX \equiv_r \bY$ \label{equiv}, if 
  $$
  \EE[f(\bX)] = \EE[f(\bY)] \quad \forall f \in \cp_r.
  $$
  Moreover, we say that a sequence $(\bX^n)_{n\geq 0} \subset \APr(\tops)$ converges to $\bX \in \APr(\tops)$ in 
  \begin{enumerate}
  \item the extended weak topology of rank $r$ \label{extended weak topology} if 
    \begin{align}
     \lim_{n \to \infty} \EE[f(\hat \bX^{r,n})] = \EE[f(\hat \bX^r)] \quad  \forall f \in C_b(\tops_r,\R)
    \end{align}
    where $U_r$ denotes the state space of process $\hat \bX^r$.
  \item the adapted topology of rank $r$ \label{adapted topology} if
    \begin{align}
       \lim_{n \to \infty }\EE[f(\bX_n)] = \EE[f(\bX)] \quad \forall f \in \cp_r. 
    \end{align}
  \end{enumerate}
  The extended weak topology on $\APr(\tops)$ is denoted by $\hat \tau_r$ and the adapted topology of rank $r$ by $\tau_r$.
\end{definition}
In Section~\ref{sec: top and sig} we show that \[(\APr(\tops), \tau_r)=(\APr(\tops), \hat \tau_r)\] whenever $\tops$ is compact but that for  non-compact subsets $\tops$ of Banach spaces, $\tau_r$ is in general coarser than $\hat \tau_r$; that is $\tau_r \subsetneq \hat \tau_r$.

\section{(Expected) Signatures of Rank $r$}\label{sec: sig}
In the previous Section~\ref{sec: adapted topology} we have introduced two topologies on $\APr(\tops)$, $\tau_r$ and $\hat \tau_r$.
We expect both to capture more or less the same structure (except for some subtle issues when $\tops$ is non-compact).
However, an attractive property of the extended weak topology $\hat \tau_r$ of rank $r$ is that it is specified by classical weak convergence of a stochastic process, namely weak convergence of the rank $r$ prediction process $\hat \bX^r$.
For $r=0$ it is known that weak convergence of a stochastic processes -- such as the prediction process $\hat \bX^0$ -- can be characterized as convergence of the expected signatures, \cite{Oberhauser18}.
This suggests that a similar approach can be fruitful in capturing the weak convergence of the higher rank prediction processes $\hat \bX^r$. 

Unfortunately, for $r \ge 1$ the rank $r$ prediction processes evolve in very large state spaces (of laws) that have a rich and nested structure which makes the use of expected signatures less straightforward. 
In this section we introduce higher rank (expected) signatures that are capable of capturing the law of such processes and their nested structure.
The key is to think about so-called higher rank paths that arise by currying multi-parameter paths.

\subsection{Recall: Moment Sequences of Random Variables} \label{sec:recall duality}
Before we discuss signatures it is instructive to briefly revisit classical moment sequences and fix some notation. 
\subsubsection{Moments and Duality}
Recall that for any compact set $\tops \subseteq \Ban=\R^d$, the moment map 
\begin{align}\label{eq:mmmap}
\Meas{\tops} \hookrightarrow {\prod_{m\geq0} \Ban^{\otimes m}},\quad \mu \mapsto \left(\int x^{\otimes m} \mu(dx)\right)_{m \ge 0}   
\end{align}
is an injection from the space $ \Meas{\tops} $\label{Meas} of signed Borel measures on $ \tops $ to $ \prod_{m\geq0} \Ban^{\otimes m} $.
A short way to prove the injection \eqref{eq:mmmap} is to recall that the dual of $\Meas{\tops}$ is the space of $C_b(\tops,\R)$. 
Under this duality, injectivity of the map~\eqref{eq:mmmap} amounts to the density of monomials in $C_b(\tops,\R)$  and the latter follows immediately by the Stone--Weierstrass Theorem.
Although this is not how the proof that moments can characterize laws is usually presented, this approach is very powerful when one tries to develop a similar argument on non-compact spaces, see~\cite{Oberhauser18}.
This duality is also the main reason to work with the linear space $\Meas{\tops}$ although we are ultimately interested in the convex set $\pr(\tops)$ of probability measures.

In particular, when restricted to the set of probability measures $\pr(\tops) \subset\Meas{\tops}$\label{pr}, the injection \eqref{eq:mmap} shows that the law of a $\tops$-valued random variable $X$, $\mu(\cdot)= \Prob(X \in \cdot)$, is characterized as an element of $\prod_{m\geq0} \Ban^{\otimes m}$,
\begin{align}
  \left(\EE\left[\frac{X^{\otimes m}}{m!}\right]\right)_{m \ge 0} \in  \prod_{m\geq0} \Ban^{\otimes m}.
\end{align}
Note that above we have included the factorial decay $m!$.
This is convenient since these terms arise in the Taylor expansion of the exponential function.
In the case of compactly supported random variables this does not make a difference but much more care need to be taken in the non-compact case and we return to this discussion in Section~\ref{sec:tensor normalization}.  

\subsubsection{Tensors, Moments, Exponentials}\label{sec:tensors}
The tensor exponential provides a concise, coordinate-free way of expressing moment relations. 
\begin{definition}
Let $V$ \label{bs} be a Banach space. Define the exponential map  
\[\exp: \Ban \mapsto {\prod_{m\geq0} \Ban^{\otimes m}}, \quad v \mapsto \left(\frac{v^{\otimes m}}{m!}\right)_{m\ge 0}.\] 
\end{definition}
To get used to this notation, it is instructive to apply the above exponential map with $V=\R^d$: spelled out in coordinates, the exponential map reduces to  the usual moment map,
\begin{align}
  x=(x^i)_{i=1,\ldots,d} \in \R^d \mapsto \left(\frac{x^{\otimes m}}{m!}\right) \simeq \left(\frac{x^{i_1} \cdots x^{i_m}}{m!}\right)_{ 1 \leq i_1,\ldots,i_m\leq d }.
\end{align}
Applied to a random variable $X$ taking values in a compact subset $\tops \subset \R^d$, all the moments of $X$,
\begin{align}\label{eq:mmap}
  \EE[\exp X]=(1, \EE[ X], \frac{1}{2!}\EE[ X^{\otimes 2}], \ldots)_{m \ge 0} \in  {\prod_{m\geq0} \Ban^{\otimes m}} 
\end{align}
are given as the expected value of the $ {\prod_{m\geq0} \Ban^{\otimes m}}$-valued random variable $\exp(X)$.
Weak convergence is then characterized as convergence of the expected value of the tensor exponential.
\begin{proposition}\label{prop: weak convergence and moments for random variable}
  Let $(X_n)$ be a sequence of random variables that take values in a compact subset $\tops \subset \R^d$.
  Then $X_n$ converges weakly to a random variable $X$ if and only if
 \begin{align}
 \EE[\exp X_n] \rightarrow \EE[\exp X] \text{ as } n \to \infty
\end{align}
where convergence on $ \prod_{m\geq0} (\R^d)^{\otimes m}$ is defined as convergence on each degree $(\R^d)^{\otimes m}$.
\end{proposition}
\begin{proof}
  The assumption of compact support implies tightness, hence the statement follows by Prohorov's theorem if one shows that if $(X_n)$ converges weakly along a subsequence to $Y$, then $Y$ equals $X$ in law. 
  But if $X_{n_k} \rightarrow Y$ weakly as $k \rightarrow \infty$, then by assumption $\lim_k\EE[p(X_{n_k})] = \EE[p(Y)]$ for any polynomial $ p $.
  The assumption also implies that $\lim_{n}\EE[p(X_n)] = \EE[p(X)]$, hence
  \[
\EE[p(X)]= \EE[p(Y)]
  \]
  for any polynomial $p$. Since polynomials are dense in $C(\tops,\R)$, this implies that $\Law{X}=\Law{Y}$. 
\end{proof}
To put the above into the context of the rest of this paper, note that these results can be reformulated as saying that the topology of weak convergence on the space of random variables that take values in a compact state space $\tops \subset \R^d$ is the initial topology of the map  
\begin{align}
  (\Omega,\cF,X) \mapsto \varphi \left( (\Omega, \cF, X) \right) \coloneqq \left( \EE \left [ \frac{X^{\otimes m}}{m!} \right] \right)_{m \ge 0} \in {\prod_{m\geq0} \Ban^{\otimes m}}, 
\end{align}
resp.~the weak topology is induced by the (semi-)metric
\begin{align}
  (\Omega_1,\cF,X) \times (\Omega_2,\cG,Y) \mapsto
  \| \varphi( (\Omega_1, \cF, X)) -  \varphi( (\Omega_2, \cG, Y)) \|.
\end{align}

To derive the analogous statement for stochastic processes, the first step is to find a suitable replacement for the tensor exponential $\exp$ to lift a path into $ \prod_{m\geq0} \Ban^{\otimes m}$ -- this leads to the notion of the signature.
\subsection{Signatures as Non-Commutative Exponentials}\label{sec:signatures}
To apply a similar reasoning to paths rather than vectors, one needs to take the sequential order into account as time progresses.
To do so we use that the linear space of tensors, $\prod_{m\geq 0} \Ban^{\otimes m}$, carries a natural non-commutative product.
\begin{definition}
For $s=(s_m)_{m\geq 0}, t=(t_m)_{m\geq 0} \in \prod_{m\geq 0} \Ban^{\otimes m}$ define 
\begin{align}\label{eq:tensor product}
  s \cdot t \coloneqq \left(\sum_{i=0}^m s_it_{m-i}\right)_{m \ge 0} \in \prod_{m\geq0} \Ban^{\otimes m}. 
\end{align}
We refer to $s \cdot t$ as as the so-called {tensor convolution product} of $s$ and $t$.
\end{definition}
To account of the sequential order in a path $x(0),x(1),\ldots,x(T)$, we now simply lift the increment of a path $x(t+1)-x(t)$ at time $t$ into $\prod_{m\geq0} \Ban^{\otimes m}$ via $\exp(x(t+1)-x(t))$, and then use the tensor convolution product~\eqref{eq:tensor product} to ``stitch these lifted increments together''.
For our purposes it turns out to be useful to first augment a path with an additional time-coordinate, that is instead of increments $x(t+1)-x(t) \in \Ban$ we consider increments \[\Delta_t x\coloneqq(t+1,x(t+1)) - (t,x(t)) = (1, x(t+1)-x(t)) \in \R \oplus \Ban\] and use the tensor exponential to embed these increments into $\prod_{m\geq0} (\R\oplus\Ban)^{\otimes m}$.
A final but important observation is that it is better to work with a slightly smaller space $\Ta{1}{\Ban} \subset  \prod_{m \ge 0} (\R \oplus \Ban)^{ \otimes m}$.
The main reason is that on this smaller space one can canonically lift a norm on $\Ban$ to a norm on $\Ta{1}{\Ban}$; see Appendix \ref{app:tensor algebras} and \ref{app: signature, feature map and MMD} for details on $\Ta{1}{\Ban}$. 

Putting everything together results in the definition of the (discrete time) signature,
\begin{definition}\label{def:sig1}
  Let $ \Ban $ be a Banach space and $\idx=\{0,1,\ldots,T\}$.
  The (rank $1$) signature map is defined as
\begin{align}
\Sigo:(\idx \to \Ban\big) \label{paths} \rightarrow \Ta{1}{\Ban}, \quad  x \mapsto \prod_{t \in \idx}\exp \Delta_tx.
\end{align}
where $\Delta_0x \coloneqq (1, x(0))$ and $\Delta_tx \coloneqq (1,x(t)-x(t-1)) \in \R \oplus \Ban$.
The (rank $1$) expected signature map is defined as
\begin{align}
	\esigo:\Meas{ \idx \to \Ban} \rightarrow \Ta{1}{\Ban}, \quad    
	\mu \mapsto \int_{x \in \Ban^{\idx}} \Sigo(x) \mu(dx). 
\end{align}

\end{definition}
 A guiding principle is that the signature of a path is the natural generalization of the monomials of a vector; resp.~the expected signatures of a stochastic process is the natural generalization of the moment sequence of a vector-valued random variable.
 Indeed, by taking $|\idx|=2$, the above equation recovers the classical tensor exponential exponential (with one additional coordinate for time), \[\Sig{x_1,x_0} =\exp(\Delta x) = \exp( (x_1-x_0,1) )  = 1 + \Delta x + \frac12 (\Delta x)^{\otimes 2} + \frac16 (\Delta x)^{\otimes 3} + \cdots.\]
 From this point of view, the following Theorem is then not surprising.
\begin{theorem} \label{cor:injective}
  Let $\Ban$ be a Banach space and $\tops \subset \Ban$ compact. 
  \begin{enumerate}
  \item \label{itm:inj}
    The map $\mathrm S: \big(\idx \to \Ban\big) \rightarrow \Ta{1}{\Ban}$ is injective.
  \item \label{itm:dense} 
    The family of linear signature functionals
    \begin{align}
      \{ x \mapsto \langle l,\Sig{x} \rangle \, \colon \, l \in \bigoplus_{m\geq 0} (\Ban^{\otimes m})^\star \} 
    \end{align}
    is dense in $ \mathrm C(\idx \to \tops,\R) $ with the uniform norm.
    \item \label{itm:char} 
      The map $\esigo: \Meas{\idx \to \tops} \rightarrow \Ta{1}{\Ban}$ is injective. 

  \end{enumerate}
\end{theorem}
\begin{proof}
  Item~\ref{itm:inj} is a special case of  \cite[Theorem 1]{Chen58}.
  Item~\ref{itm:dense} is due to Fliess \cite[Corollary 4.9]{fliess76}.
  Item~\ref{itm:char} follows since if $ \mu, \nu \in \Meas{\idx \to \tops} $ are such that $ \esigo(\mu) = \esigo(\nu) $, then $ \langle \esigo(\mu), \ell \rangle = \langle \esigo(\nu), \ell \rangle $ for any $ \ell \in \bigoplus_{m\geq 0} (\Ban^{\otimes m})^\star $ so by Item~\ref{itm:dense} it holds that $ \mu(f) = \nu(f) $ for any $ f \in \mathrm C(\idx\to\tops,\R) $, hence $ \mu=\nu $.
\end{proof}

\begin{remark}\label{rem:classical}
  Everything in this section is classical: our discrete signature coincides with Chen's \cite{Chen54} iterated integral signature, that is $\Sig{x}_m = \int dx^L_{t_1} \otimes \cdots \otimes dx^L_{t_m}$ where $x:[0,T] \rightarrow \Ban$ denotes the path given by linear interpolation of $\{(t,x(t)): t \in \idx\}$. 
  Usually, signatures are defined without the time--coordinate and only capture the path up to re-parametrization, but the adapted topologies depend on the parametrization so it is natural to include the time--coordinate.
  Nevertheless, the results in the following sections can be easily adapted without the additional time-coordinate and it might be interesting to study the resulting adapted topology for equivalence classes of un-parametrized paths; see \cite{Oberhauser18} for a discussion for the case of weak convergence, $r=0$.
  See also Appendix~\ref{app: signature, feature map and MMD} for more on signatures.
\end{remark}

\subsection{Paths and Signatures of Higher Rank}\label{sec:higher rank signatures}
Section~\ref{sec:signatures} recalled that the [expected] signature can characterize [measures on] paths.
Our goal is to characterize the predictions processes introduced in Section~\ref{sec: adapted topology}.
Simply applying the expected signature to a prediction process would ignore the nested structure of the state spaces of these processes, see Definition~\ref{def: higher rank prediction process}, and we heavily use this structure in the proof of our main result, Section~\ref{sec: top and sig}.
To address this we first introduce higher rank paths which formalize paths evolving in spaces of paths and then use this to introduce higher rank [expected] signatures.

\begin{definition}
Let $(\idx_r)_{r\geq1}$ be a sequence of finite ordered sets and $\tops$ a topological space. Let $\tops_0:=\tops$ and define $(\tops_r)_{r \ge 0}$\label{topsr} inductively, 
\[\tops_r:= \big(\idx_r \to \tops_{r-1}\big) \]
We refer to an element of $\tops_r$ as a path of rank $r$ in the state space $\tops$
\end{definition}
Explicitly, these spaces can be unravelled as 
\begin{align}\label{eq:higher rank paths}
	\tops_r=\idx_r \rightarrow \tops_{r-1}=(\idx_r \rightarrow \underbrace{(\idx_{r-1} \rightarrow \cdots \underbrace{(\idx_{2} \rightarrow \underbrace{(\idx_1\rightarrow \tops)}_{\tops_1})}_{\tops_2}\cdots)}_{\tops_{r-1}}).
\end{align}
A rank $1$ path coincides with the usual definition of a path from $\idx_1$ into $\tops$.
Evaluating a rank $r$ path at time $t_r \in \idx_r$ yields a rank $r-1$ path in the same state space, that is for $x \in \tops_r$, $x(t_r) \in \tops_{r-1}$ for every $t_r \in \idx_r$.

Recall that the Signature from Definition~\ref{def:sig1} injects any path evolving in a Banach space $\vs$ into the Banach space $\Ta{1}{V}$.
By iterating compositions of this map and defining $ \Ta{r+1}{\vs} $ \label{rtensor} inductively as the completion of $ \bigoplus_{n\geq 0} \Ta{r}{\vs}^{\otimes n} $ with respect to a suitable norm (see Appendix \ref{app:tensor algebras} for details) we get the following.
\begin{definition}\label{def: higher rank signature}
Let $ \Ban $ be a Banach space. Define the family of maps $(\sigr{r})_{r\ge 1}$\label{sigr}, 
\begin{align}
\sigr{r}: \vs_r
&\rightarrow \Ta{r}{\vs}
\end{align}
inductively by setting $\sigr{1}\coloneqq \Sigo$ and for $r \geq 2$,
\begin{align}
  \sigr{r}(x) := \Sig{x^\star \sigr{r-1}}, 
\end{align}
where $x^\star \sigr{r-1}$ denotes the pullback\footnote{that is $(x^\star \sigr{r-1})(t):= \sig{r-1}{x(t)}$ using that $x \in \vs_r$ and $x(t) \in \vs_{r-1}$.} of $\sigr{r-1}$ by $x$. We call $\sigr{r}$ the signature map of rank $r$. 
\end{definition}
\begin{figure}
  \centering
	\begin{tikzpicture}
	\matrix (m) [matrix of math nodes,row sep=4em,column sep=3em,minimum width=3em]
	{
		\, & \overbrace{\big(\idx_r \rightarrow \vs_{r-1}\big)}^{\vs_r} & \, \\
		\idx_r \rightarrow \Ta{r-1}{\vs} & \, & \Ta{r}{\vs} \\};
	\path[-stealth]
	(m-1-2) edge node [above] {$\sigr{r-1}\,\,\,\,$} (m-2-1)
	(m-2-1) edge node [below] {$\mathrm{S}$} (m-2-3)
	(m-1-2) edge [dashed] node [above] {$\sigr{r}$} (m-2-3);
	\end{tikzpicture}
	\caption{The inductive definition of $ \sigr{r} $. By extending the map $ \sigr{r-1} $ to a map $ \vs_r \rightarrow (\idx_r \rightarrow \Ta{r-1}{\vs}) $, the signature $ \mathrm{S} $ can be applied to it to form $ \sigr{r} : \vs_r \to \Ta{r}{\vs} $.}
\end{figure}
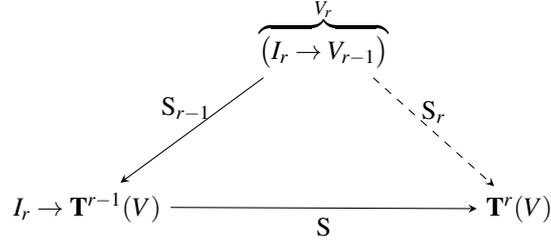
\begin{example}
  Schematically, we can think of rank $r$ signatures and  rank $r$ paths as 
  \begin{align}\label{higher rank paths}
    \vs_r= (\idx_r \rightarrow \underbrace{ (\idx_{r-1} \rightarrow \cdots \underbrace{ (\idx_{2} \rightarrow \underbrace{(\idx_1\rightarrow \vs)}_{\vs_1\hookrightarrow \Ta{1}{\vs}})}_{\vs_2\hookrightarrow \Ta{2}{\vs} }\cdots)}_{\vs_{r-1}\hookrightarrow \Ta{r-1}{\vs}}).
  \end{align}
  The above construction starts with applying the usual signature to the innermost bracket to turn the map $V_1 \to  I_1 $  into an element of $\Ta{1}{V}$ (the top curly bracket).  
  The next step turns the map $I_2 \to \Ta{1}{V}$ into an element of $\Ta{2}{V}$, etc.
  It is instructive to go through a couple of case for $r$. 
  \begin{enumerate}\item For $r=1$, we are given a (rank 1) path $x:\idx_1 \rightarrow \vs$, and $\sig{1}{x}$ is by definition the signature of $x$, $\sig{1}{x} \in \Ta{1}{\vs}$.
    That is, $\sigr{1}$ maps rank $1$ paths in the state space $\vs$ to elements of $\Ta{1}{\vs}$.

\item For $r=2$, we are given a rank $2$ path $x$ in the state space $\vs$, $x:\idx_2 \rightarrow (\idx_1 \rightarrow \vs)$. The evaluation of $x$ at any $t_2 \in \idx_2$ yields a rank $1$ path in the state space $\vs$
\begin{align}
	x(t_2): \idx_1 \rightarrow \vs, \quad t_1 \mapsto x(t_2)(t_1).  
\end{align}
Since $\sigr{1}$ maps a rank $1$ path in the state space $V$ to an element of $\Ta{1}{\vs}$, the pullback of $\sigr{1}$ by $x^\star$ equals 
\begin{align}\label{eq: r=2 pullback}
	x^\star \sigr{1}: \idx_2 \rightarrow \Ta{1}{\vs},\quad t_2 \mapsto \sigr{1} (x(t_2)). 
\end{align}
By definition, $\sig{2}{x}$ is the signature of this rank 1 path, $x^\star \sigr{1}$, that evolves in the state space $\Ta{1}{\vs}$, 
\begin{align}
	\sig{2}{x}= \Sig{x^\star \sigr{1}}, 
\end{align}
and therefore $\sig{2}{x} \in \Ta{2}{\vs} \subseteq \prod_{m\geq 0} (\R \oplus\Ta{1}{\vs})^{\otimes m}$.
That is, $\sigr{2}$ maps rank $2$ paths in the state space $\vs$ to elements of $\Ta{1}{\vs}$.
\end{enumerate}
\end{example}
\begin{proposition} \label{prop: higher injective}
  The map $\sigr{r}: \vs_r \rightarrow \Ta{r}{\vs}$ is injective.
\end{proposition}
\begin{proof}
	Follows by iterating Theorem \ref{cor:injective}.
\end{proof}

\subsection{Measures and Expected Signatures of Higher Rank}\label{sec: higher rank expected signature}
Our goal is to inject the laws of predictions processes into a normed space.
Recall that the laws of prediction processes have a rich nested structure, for example
\begin{align*}
  \Law{\hat \bX^0} &\in \Meas{ \idx \rightarrow \vs}\eqqcolon \cM_1,\\
\Law{\hat \bX^1}&\in \Meas{ I \rightarrow \Meas{\idx \rightarrow \vs}} \eqqcolon \cM_2, \\
\Law{\hat \bX^3} &\in \Meas{ I \rightarrow \Meas{ \idx \rightarrow \Meas{\idx \rightarrow \vs}}} \eqqcolon \cM_3.  
\end{align*}
Capturing this nested structure is essential when discussing the adapted topologies.  

\begin{definition}
Let $\idx_1,\ldots,\idx_r$ be finite ordered sets and $\tops$ a topological space.
For $r=1$ define $\cM_{0}\coloneqq \pr_0 \coloneqq \tops$ and for $ r \ge 1  $
\begin{align}
  \cM_{r}(\tops)&:= \Meas{\idx_r \to \cM_{r-1}(\tops)},\\
  \cP_{r}(\tops)&:= \prob{\idx_r \to \cP_{r-1}(\tops)}
\end{align}
We endow $\cM_{r}(\tops)$ and $\cP_{r}(\tops)$ with the natural weak topology. 
We refer to an element of $\cM_{r}(\tops)$\label{cmr} as a rank $r$ measure on $\tops$ and an element of $\cP_r(\tops)$\label{prr} as a rank $r$ probability measure on $\tops$. 
\end{definition}
Clearly, $\cP_r(U) \subset \cM_r(U)$ and these spaces can be written explicitly as 
\begin{align}\label{eq:higher rank measure}
\cM_{r}(\tops)&=\Meas{\idx_r \rightarrow \Meas{\idx_{r-1} \rightarrow \cdots \Meas{\idx_{2} \rightarrow \Meas{\idx_1\rightarrow \tops}}\cdots} },\\
\cP_r(\tops) &=  \prob{ I_r \to \prob{ I_{r-1} \to \cdots \prob{I_2 \to \prob{I_1 \to V}}\cdots } }.
\end{align}
As mentioned in Section~\ref{sec:recall duality}, although we are interested in the convex set of probability measures $\cP_r$, working with the larger linear space of Borel measures $\cM_r$ allows us to use duality arguments.
We emphasize that $\cM_{r}(\tops)$ is significantly bigger than $\Meas{\tops_{r}}$: the latter embeds into the former by taking the $r-1$ innermost measures in the parenthesis in~\eqref{eq:higher rank measure} to be Dirac measures.  

Analogous to how we iterated signature maps and tensor algebras in the previous section, we now construct expected signatures to provide an injection $\cM_{r}(\vs) \hookrightarrow \Ta{r}{\vs}$.
\begin{definition}
  Let $\vs$ be a Banach space.
  Define the family of maps $(\esig{r})_{r \ge 1}$\label{esigr} inductively by setting $\esig{0}:=\operatorname{id}_{\vs}$ and (whenever the integral is well--defined)
  \begin{align}
    \esig{r}: \cM_{r}(\Ban)  \rightarrow \Ta{r}{\vs}, \quad
    \mu \mapsto \int \Sig{x^\star\esig{r-1}} \mu(\mathrm{d}x),
  \end{align}
  where $x^\star\esig{r-1} $ denotes the pullback of $\esig{r-1}$ by $x$. 
  We call $\esig{r}$ the expected signature map of rank $r$. 
$  $\end{definition}
The following Proposition~\ref{prop:higher rank exp} generalizes that expected signature characterizes laws of processes (Theorem~\ref{cor:injective} item \ref{itm:char}).
We postpone its proof to Theorem~\ref{thm:main}.
\begin{proposition}\label{prop:higher rank exp}
  Let $ \vs $ be a separable Banach space and $ \cK \subset \vs $ compact.
  Then \[\esig{r}:\cP_r(\cK) \rightarrow \Ta{r}{\vs}\] is injective.
\end{proposition}

\begin{example}
  It is instructive to run through the first few iterations of $r$ for $\esig{r}$.
  Since one can always assume that the process $X = (X_t)_{t \in I}$ is the canonical coordinate process defined on the probability space $\big((I \to \cM_{r-1}(\vs)), \mu\big)$, we may also write $\ESig{r}{\mu} = \EE_{\mu}[S \circ \ESig{r-1}{X}]$.
  \begin{itemize}
  \item 
If $r = 1$, then for any probability measure $\mu \in \cP_1(\vs) \subset \cM_{1}(\vs) = \cM(I \to \vs)$, the mapping $\esig{1}(\mu) = \EE_{X \sim \mu}[S(X)]$ is the expected signature of the discrete--time stochastic process $X = (X_t)_{t \in I}$ with law $\mu$.
\item 
If $r = 2$, then for any probability measure $\mu \in \cM_{2}(\vs) = \cM\big(I \to \cM_{1}(\vs)\big)$, fix some stochastic process $X = (X_t)_{t \in I}$ with values in $\cM_{1}(\vs)$ and law $\mu$. For any $t \in I$, $X^\star \esig{1}(t) = \esig{1}(X(t))$ is the expected signature of $X(t)$; and hence $X^\star \esig{1}$ can be thought of as a stochastic process taking values in the vector space $\Ta{1}{\vs}$ and we may compute its expected signature. 

For a particular example of this,  if $Z = (Z_t)_{t \in I}$ is a discrete--time process taking values in $\vs$ defined on some stochastic basis $(\Omega,(\mathcal{F}_t),\PP)$, then $X_t := \PP[Z \in \cdot | \mathcal{F}_t]$ is a regular conditional distribution of $Z$ given $\mathcal{F}_t$. Let $\mu = \mathcal{L}(X)$ be the law of the measure-valued process $X$, then
\begin{align}
\esig{2}(\mu) = \EE[S(t \mapsto \EE[S(s \mapsto Z_s)|\mathcal{F}_t])].
\end{align}
We will give a complete description of $\esig{r}$ for this special case in Section~\ref{sec: top and sig}.
\end{itemize}
\end{example}

\subsection{Non-Compactness and Robust (Expected) Signatures}\label{sec:tensor normalization}
Even for random variables in $\R^d$, elementary examples show that the sequence of moments does not characterize the law when their support is non-compact; in particular, Proposition \ref{prop: weak convergence and moments for random variable} is not true without compact support.
The same applies to stochastic processes and their (higher rank) expected signatures.

The ``robust (Signature) Moments'' construction from~\cite{Oberhauser18} yields an extension of the injectivity of the higher rank expected signature from the previous sections to paths in general (non-compact) Banach spaces.
We emphasize that the results in Section~\ref{sec: top and sig} are already interesting for the case of compact state spaces that we have discussed in the previous sections and we invite readers less familiar with signatures to skip this section. 

\begin{proposition}\label{prop: rank r exp sig injective}
  \label{prop: higher norm injective}
  Let $ \Ban $ be separable Banach space.
  For every $r\ge 0$ there exist maps 
\begin{align}
  \sigrno{r} &: \vs_r \rightarrow \Ta{r}{\vs}\\
  \esigrno{r}&:\cM_r(\vs) \rightarrow \Ta{r}{\vs}
\end{align}
that are both bounded, continuous and injective. 
Further, the space $ \Ta{r}{\vs} $ is also a separable Banach space.
We refer to $\sigrno{r}$ as the robust signature map of rank $r$ and to $\esigrno{r}$ as the robust expected signature map of rank $r$.
\end{proposition}
\begin{proof}The fact that every space $\Ta{r}{\vs}$ for $r \ge 1$ is a separable Banach space follows immediately from Definition \ref{def:norm}. Then we can show that on each $\Ta{r}{\vs}$ there exists a so--called tensor normalization map $\Lambda$ with codomain the unit ball of $\Ta{r}{\vs}$ such that the composition $\Lambda \Sigo := \Lambda \circ \Sigo$ preserves the algebraic properties of the signature map. We refer to \ref{app:normalization} for details on the construction of the normalization $\Lambda$.
 Let $\mu$ be an element of $\pr_{r}(\vs) = \pr(\idx \to \pr_{r-1}(\vs))$. Let $ Z^{r-1} $ be a stochastic process with values in $ \pr_{r-1}(\vs) $ and law $ \mu $, then we define for every $r \ge 1$, $\sigrno{r}(Z^{r-1}):= \Lambda \Sigo \circ (t \mapsto \esigrno{r-1}(Z^{r-1}_t))$ and
\begin{equation}\label{eq: equality for push forward measures}
\esigrno{r}(\mu) = \EE_{Z^{r-1}\sim\mu}[\sigrno{r}(Z^{r-1})]
\end{equation}
with the convention $ \esigrno{0}(Z^0) = Z^0$.
If $r = 1$, then this follows from Proposition \ref{prop: expected signature is injective} in Appendix  \ref{app:normalization}.
By the induction hypothesis, $\esigrno{r-1}$ is continuous and injective, hence the assertion about injectivity follows immediately from Proposition \ref{prop: central proposition}. 
Finally, the continuity of $\esigrno{r}$ follows from the Skorokhod representation theorem since if $\Ban$ is separable, then $\idx \to \vs$ is separable, and hence $\pr_{r-1}(\vs)$, is separable with respect to the weak topology.
\end{proof}
For ease of notation we are going to redefine the symbols $\sigr{r}$ and $\esigr{r}$ for the remainder of the article.
\begin{definition}
  Let $\Ban$ be a separable Banach space, $\tops \subset \Ban$, and $r \ge 0$.
  For the rest of this article we denote
  \begin{align}
    \sigr{r}: \tops_r &\to \Ta{r}{\Ban},\quad x \mapsto \sigrn{r}{x}, \\ 
    \esigr{r}:\cM_r(\tops) &   \to \Ta{r}{\Ban},\quad \mu \mapsto \esigrn{r}{\mu}.
  \end{align}
\end{definition}

\section{The Adapted Topology and Higher Rank Signatures} \label{sec: top and sig}
Our leitmotif is that expected signatures can be regarded as a generalization of the classical moment map.
Indeed, for $r=0$ we have by definition that $\bX=\hat \bX^0$ and the initial topology of the the map \[\APr \to \Ta{1}{\Ban},\quad \bX \to \ESig{1}{\operatorname{Law}(\hat \bX^0)}\] 
is the topology of weak convergence $(\APr,\tau_0)$, see~\cite{Oberhauser18}.
This suggests that, at least locally, the initial topology of the map \[ \APr \to \Ta{r}{\Ban},\quad \bX \to \ESig{r+1}{\operatorname{Law}(\hat \bX^r)}\] is the rank $r$ adapted topology $(\APr_r,\tau_r)$.
In this Section we show that this is indeed true in great generality.
\begin{definition}\label{distance}
  Let $\Ban$ be a separable Banach space and $\tops \subset V$.  
  For $r \ge 0 $ define 
  \begin{align}
    \Phi_r: \APr(\tops) \to \Ta{r+1}{V},\quad \bX \mapsto \ESig{r+1}{\Law{\hat \bX^r}} 
  \end{align}
  and
	\begin{align}
    d_r: \APr(\tops) \times \APr(\tops) \to [0,\infty), \quad (\bX, \bY) \mapsto \lVert \Phi_r(\bX) - \Phi_r(\bY) \rVert_{r+1}.
	\end{align}
\end{definition}
Our main result is 
\begin{theorem}\label{thm:mainA}
  Let $\Ban$ be a separable Banach space and $\tops \subset \Ban$ compact.  
  The following topologies on $\APr(\tops)$ are equal
\begin{enumerate}
\item the adapted topology of rank $r$, $\tau_r$, 
\item the extended weak topology of rank $r$, $\hat \tau_r$,
\item the initial topology of the map $\Phi_r: \APr(\tops) \to \Ta{r}{\Ban}$, 
\item the topology induced by convergence in the semi-metric $d_r$ on $\APr(\tops)$. 
\end{enumerate}
Moreover, the same statement holds locally if $\tops$ is not compact; see Theorem~\ref{thm:main}.  
\end{theorem}
Restricted to $r=1$ and processes with their natural filtration, the semi-metric $d_1$ adds another entry to the list of distances that induce the adapted topology $\tau_1$.
However, even for this $r=1$ case, the characterization of the adapted topology as the initial topology of a map into a normed, graded space rather than the topology induced by a (semi-)metric, is to the best of our knowledge new.
\begin{corollary}[\cite{Beiglbock19}] \label{corollary:In natural filtration case}
Let $\tops$ be as in Theorem~\ref{thm:mainA} and denote by $\APr_{\text{Natural}}(\tops) $ the subset of $\APr(\tops)$ of processes equipped with their natural filtration. 
Then the following topologies on $\APr_{\text{Natural}}(\tops) $ are equal 
		\begin{itemize}
      \item the topology induced by $d_1$,
			\item the topology induced by adapted Wasserstein distance,
			\item the topology induced by symmetrized-causal Wasserstein distance,
			\item Hellwig's information topology,
			\item Aldous' extended weak topology,
			\item the optimal stopping topology.
		\end{itemize}
\end{corollary}

The remainder of this Section is devoted to the proof of Theorem~\ref{thm:mainA}. 

\subsection{Higher Rank Conditional Signature Process.}
The domain of $\esig{r}$ is all of $\cM_{r}(\Ban)$.
When restricted to the laws of prediction processes, this additional structure yields an useful interpretation in terms of conditional expectations; e.g.~for $r=1$ and $t \in \idx$,
\begin{align}\label{eq: induction start}
\ESig{1}{\hat \bX^1_t} = \int \Sig{x} \PP[X \in dx| \mathcal{F}_t]= \EE[\Sig{X}| \cF_t].
\end{align}
This motivates the following definition
\begin{definition}
	Let $\bX=(\Omega,(\cF_t),\mathbb{P},X) \in \APr(\vs)$. 
	We define a family of adapted processes $(\bar \bX^r)_{r \ge 0 }$\label{barbxr} by $\bar \bX^{r}=(\Omega,\cF, \mathbb{P}, \bar X^r)$ with $\bar X^r$ given inductively as 
	\begin{align}
	\bar X^{r}_t:= \EE[ \Sig{\bar \bX^{r-1}}|\cF_t] 
	\end{align}
	and $\bar \bX^0_t= X_t$.
	We call $\bar \bX^r$ the rank $r$ conditional signature process of $\bX$.
\end{definition}
\begin{proposition}\label{prop: concrete description of expected signature of higher rank}
	For every $r \ge 1$ and $\bX \in \APr(\vs)$ it holds that 
	\begin{align}\label{eq: target equation}
	\ESig{r}{\hat \bX^r_t}= \bar \bX^{r}_t \quad \forall t \in \idx. 
	\end{align}
	In particular, \[\Phi_r(\bX) \equiv \ESig{r+1}{\Law{\hat \bX^{r}}}= \EE\bar \bX^{r+1}_0.\]
\end{proposition}
\begin{proof}
	The second claim follows immediately from~\eqref{eq: target equation} since
	\begin{align}
	\EE\bar \bX^{r}_t = \EE \ESig{r}{\hat \bX^r_t} = \EE \int \ESig{r}{x} \PP[\hat \bX^{r-1} \in dx| \mathcal{F}_t] = \int \ESig{r}{x} \PP[\hat \bX^{r-1} \in dx] = \ESig{r}{\Law{\hat \bX^{r-1}}}.
	\end{align}
	For the proof of~\eqref{eq: target equation} we proceed by induction over $r \ge 1$.
	The starting case, $r=1$, is given in~\eqref{eq: induction start}.
	For the induction step, assume that~\eqref{eq: target equation} holds true for some $r \ge 1$. 
	We denote by $\mu_r$ the measure
	\begin{align}
	\mu_r = \Prob(\hat X^{r} \in \cdot | \mathcal{F}_t).
	\end{align}
	By definition of $\esig{r+1}$ we see that 
	\begin{align}
	\ESig{r+1}{\hat \bX^{r+1}_t}  = \int \Sig{x^{\star} \esig{r}} \mu_r(dx)
	= \EE[ \Sig{ s \mapsto \ESig{r}{\hat X^r_s}}|\cF_t] = \EE[ \Sig{ s \mapsto \bar \bX^r_s}|\cF_t]
	\end{align}
	where we used the induction hypothesis, $\ESig{r}{\hat X^r_s} = \bX^r_s$ in the last step.
\end{proof}
This interpretation of $\Phi_r(\bX)$ in terms of the rank $r$ conditional signature process $\bar \bX^{r+1}$ turns out to be very useful in the next section, in particular for the proof of Theorem \ref{thm:2}.  

\subsection{Embedding and Metrizing Adapted Topologies}\label{subsect: embedding and metrizing adapted topologies}
\begin{theorem}\label{thm:2}
  Let $\Ban$ be a a separable Banach space and $ \bX, \bY \in \APr(\vs) $.
  For every $r\ge 0$ the following are equivalent
	\begin{enumerate}
  \item \label{itm: equal mod r} $\EE[f(\bX)] = \EE[f(\bY)] \quad \forall f \in \cp_r$, 
		\item\label{itm: equal pred proc2} $\Law{\hat \bX^r} = \Law{\hat \bY^r}$,
		\item\label{itm: equal pred proc3} $\Law{\hat \bX^0, \ldots, \hat \bX^r} = \Law{\hat \bY^0,\ldots,\hat \bY^r}$. 
  \end{enumerate}
  \begin{enumerate}
    \setcounter{enumi}{3}
		\item\label{itm: equal exp sig 4} $\Phi_r(\bX) = \Phi_r(\bY)$. 
	\end{enumerate}
\end{theorem}
We prepare the proof of Theorem~\ref{thm:2} with a Lemma. 
\begin{lemma}\label{lem: approx indicator}
	For every $r \ge 0$ and every Borel set $B \subseteq (\idx \rightarrow \cM_r)$, there exists a sequence of uniformly bounded adapted functionals $f_k \in \cp_r$ such that $1_B \circ \hat \bX^r = \lim_{k \rightarrow \infty}f_k(\bX)$ in probability.
\end{lemma}

\begin{proof}[Proof of Lemma \ref{lem: approx indicator}]
	If $r = 0$, then since $\Ban$ is a Polish space and $\hat \bX^0 = X$, the claim holds due to Urysohn's lemma and Dynkin's lemma.
	
	Now consider the case $r \ge 1$. Since $ \cp_r $ is an algebra, by Dynkin's lemma it suffices to consider the case $B = B_0 \times \ldots \times B_T$, where each $B_i$ is a Borel set of $ (\idx \to \cM_{r-1}(\Ban))$. Hence we have
	\begin{align}
	1_B \circ \hat \bX^r = 1_{B_0} \circ \hat \bX^r_0 \times \ldots \times 1_{B_T} \circ \hat \bX^r_{T}.
	\end{align}
	Furthermore, since $\Meas{\idx \to \cM_{r-1}(\Ban)}$ carries the Borel $\sigma$--algebra generated by the sets of the form
	\begin{align}
	&e_U^{-1}(J) := \{\mu \in \Meas{\idx \to \cM_{r-1}(\Ban)}: e_{U}(\mu) := \mu(U) \in J\}, \\ 
	&U \text{ Borel set in } (\idx \to \cM_{r-1}(\Ban)), \quad J \subseteq [0,1],
	\end{align}
	we may use Dynkin's lemma again and assume that $B_i = e_{U_i}^{-1}(J_i)$ for some Borel set $U_i$ in $(\idx \to \cM_{r-1}(\Ban))$ and some interval $ J\subseteq [0,1]$. Now, using that $\hat \bX^r_t = \mathbb{P}(\hat \bX^{r-1} \in \cdot |\mathcal{F}_t)$, it holds that for all $t$,
	\begin{align}
	1_{B_t} \circ\hat \bX^r_t = 1_{J_t} \circ \EE[1_{U_n} \circ \hat \bX^{r-1} | \mathcal{F}_t].
	\end{align}
	By the induction hypothesis, we have 
	\begin{align}
	1_{U_n} \circ \hat \bX^{r-1} = \lim_{k \rightarrow \infty} f^n_k(\bX),
	\end{align}
	where every $f^n_{k}$ is of rank at most $ r-1 $ and is uniformly bounded, so every $\EE[f^n_{k}(\bX)|\mathcal{F}_t]$ is of rank at most $r$. Now we choose a sequence of uniformly bounded continuous functions $(\varphi_k)_{k \ge 1}$ (say, uniformly bounded by $1$) which approximates $1_{J_0} \times \ldots \times 1_{J_I}$ pointwise, so that $1_B \circ \hat \bX^r = \lim_{j \rightarrow \infty} \varphi_j(\hat \bX^r_0, \ldots, \hat \bX^r_T)$  a.s. (up to taking a subsequence if necessary). From the above observations we see that for each $j$, 
	\begin{align}
	\varphi_j(\hat \bX^r_0, \ldots, \hat \bX^r_T) = \lim_{k \rightarrow \infty} \varphi_j((\EE[f^n_k(\bX)|\mathcal{F}_t])_{t \in \idx}),
	\end{align}
	where every $\varphi_j((\EE[f^n_{k}(\bX)|\mathcal{F}_t])_{t \in \idx})$ is by definition an adapted functional of rank at most $r$. This shows that we can find a sequence of adapted functionals $(f_k)_{k\ge 1}$ of rank at most $r$, such that $1_B \circ \hat \bX^{r} = \lim_{k \rightarrow \infty} f_k(\bX)$ in probability.
\end{proof}
\begin{proof}[Proof of Theorem~\ref{thm:2}]
	\Implies{itm: equal mod r}{itm: equal pred proc2}.
	Using Lemma~\ref{lem: approx indicator}, it follows by an induction argument that $1_B \circ \hat \bX^r = \lim_{k \rightarrow \infty}f_{k}(\bX)$ implies that $1_B \circ \hat \bY^r = \lim_{k \rightarrow \infty}f_{k}(\bY)$. By \eqref{itm: equal mod r}, we have that $\EE[f_{k}(\bX)] = \EE[f_{k}(\bY)]$ for all $k\geq0$, so by the dominated convergence theorem 
	\begin{align}
	\EE[1_B \circ \hat \bX^r] = \EE[1_B \circ \hat \bY^r] \text{ for any Borel set } B
	\end{align}
	i.e.~$\Law{\hat \bX^r} = \Law{\hat \bY^r}$. 
	
	\Implies{itm: equal pred proc2}{itm: equal pred proc3}.
	For a Polish space $\mathcal{X}$, let $\MeasA{\mathcal{X}} \subset \Meas{\mathcal{X}}$ be the set of Dirac measures on $\cX$ $\{ \delta_{x} \, : \,x \in \mathcal{X} \}$.
	Define $p: \MeasA{\cX} \rightarrow \cX$ by $p(\delta_x) := x$ and note that $p$ is continuous with respect to the subspace topology on $\MeasA{\cX}$. Define 
	\begin{align}
	\pi : (\idx \to \mathcal{X}) \rightarrow \cX, \quad \pi(x)=x_T \text{ for } x=(x_1,\ldots,x_T) \in (\idx \to \mathcal{X}), \, \idx = \{1,\ldots,T\}.
	\end{align} 
	For $g: \cX \rightarrow \cX$ define $\id_{\mathcal{X}} \oplus g: \cX \rightarrow \cX^2$, as $(\id_\cX\oplus g)(x)= (x,g(x))$.
	In what follows, although the underlying space $\mathcal{X}$ may vary from line to line, we will use the same notation as above for simplicity. Since $\hat \bX^r_T = \mathbb{P}(\hat \bX^{r-1} \in \cdot|\mathcal{F}_T)$, we can write 
	\[
	\hat \bX^i_T = \delta_{\hat \bX^{i-1}} \in \MeasA{\idx \rightarrow \meas{r-1}}.
	\]
	For each $r$, define
	\[g_r:\idx \rightarrow \meas{r},\, g_r=p \circ \pi.\]
	Using that $g_r(\hat \bX^r) = \hat \bX^{r-1}$, it follows that for $r \ge 1$,
	\begin{align}
	(\hat \bX^r, \ldots, \hat \bX^{r-s}) = (\id \oplus g_{r-s+1}) \circ  (\hat \bX^r, \ldots, \hat \bX^{r-s+1}),
	\end{align}
	where $\id$ is applied to $\cX={\meas{r}^I \times \cdots\times \meas{1}^I}$.
	Since, $\id \oplus g_r$ is continuous we can iterate this composition to build a continuous function $G$, such that 
	\begin{align}
	(\hat \bX^r, \ldots, \hat \bX^0) = G(\hat \bX^r).
	\end{align}
	As a result, for any bounded continuous function $F$ defined on $\meas{r}^I \times \ldots \times \meas{0}^I$, we have 
	\[
	\EE[F(\hat \bX^r, \ldots, \hat \bX^0)] = \EE[F \circ G (\hat \bX^r)],
	\]
	Using \ref{itm: equal pred proc2} and denoting for brevity
	\[
	E:= (\idx \to \meas{r}(\cX))\times \cdots \times (\idx \to \meas{1}(\cX))
	\]
	we deduce that
	\begin{align*}
	\Law{\hat \bX^r} = \Law{\hat \bY^r} &\Rightarrow \forall F \in C_b(E), \quad  \EE[F \circ G (\hat \bX^r)] = \EE[F \circ G (\hat \bY^r)] \\
	&\Leftrightarrow \forall F \in C_b(E), \quad \EE[F(\hat \bX^r, \ldots, \hat \bX^0)] = \EE[F(\hat \bY^r, \ldots, \hat \bY^0)] \\
	&\Leftrightarrow \Law{\hat \bX^0, \ldots, \hat \bX^r}= \Law{\hat \bY^0, \ldots \hat \bY^r}.
	\end{align*}
	\Implies{itm: equal pred proc3}{itm: equal mod r}.
	We prove by induction that for any $r \ge 0$, and $f \in \cp_r$, there exists some bounded and Borel measurable $ \tilde f : (\idx \to \meas{r}(\Ban)) \rightarrow \R $ 
	\begin{align}
	f(\bX)  = \tilde{f}(\hat \bX^r).
	\end{align}
	The case $r=0$ is clear since $\hat \bX^0 = X$ so we can take $\tilde{f} = f$, which is indeed bounded and Borel measurable.
	
	For the induction step, assume the claim holds up to some $r-1$, $r \ge 2$. Then given $f \in \cp_{r-1}$, there exists a bounded Borel measurable function $\tilde{f}$ defined on $(\idx \to \meas{r-1}(\Ban))$ such that $f(\bX) = \tilde{f}(\hat \bX^{r-1})$.
	Now for every $t \in \idx $, $\EE[f(\bX) | \mathcal{F}_t]$, is an element of $\cp_r$, so by using that $f(\bX) = \tilde{f}(\hat \bX^{r-1})$, we get $\EE[f(\bX) | \mathcal{F}_t] = \EE[\tilde{f}(\hat \bX^{r-1})|\mathcal{F}_t]$.  
	
	On the other hand, since by definition $\hat \bX^r_t = \mathbb{P}(\hat \bX^{r-1} \in \cdot |\mathcal{F}_t)$ is the regular conditional distribution of $\hat \bX^{r-1}$ given $\mathcal{F}_t$, we also obtain that
	\begin{align}
	\EE[\tilde{f}(\hat \bX^{r-1})|\mathcal{F}_t] = e_{\tilde{f}}(\pi_t \circ \hat \bX^r),
	\end{align}
	where $\pi_t$ is the $t$--th coordinate mapping such that $\pi_t \circ \hat \bX^r = \hat \bX^r_t$, and $e_{\tilde{f}}$ is the evaluation map defined on $\meas{r}(\Ban)=\Meas{\idx \to\meas{r-1}(\Ban)}$ such that $e_{\tilde{f}}(\mu) := \int \tilde{f} d \mu$.
	Since $\tilde{f}$ is bounded and measurable by \cite[Corollary 7.29.1]{Bert78}, $ e_{\tilde f} $ is a bounded measurable function on $\meas{r}(\Ban)$.
	
	In other words, we have now obtained that $\EE[fX | \mathcal{F}_t] = g(\hat \bX^r)$, where $g := e_{\tilde{f}} \circ \pi_t$ is a bounded measurable mapping defined on $\mathcal{X}_r$.
	This together with the fact that $\hat \bX^{r-1}$ can be expressed as a Borel measurable function composition with $\hat \bX^r$ (see the proof of (2) $\Rightarrow$ (3)) implies that all adapted functionals of rank at most $r$ still satisfy the above claim, and completes the induction step. 
	
	\Iff{itm: equal pred proc2}{itm: equal exp sig 4}.
	By Proposition~\ref{prop: rank r exp sig injective}, $\esig{r}$ is injective on $\pr_{r}(\vs)$ hence the equivalence follows immediately from Proposition \ref{prop: concrete description of expected signature of higher rank} and the fact that $\Law{\hat \bX^r}, \Law{\hat \bY^r} \in \pr_{r+1}(\vs)$.
\end{proof}

The metrics $d_r$ (cf.~Definition \ref{distance}) locally characterize the rank $r$ extended weak topology $\hat \tau_r$.
\begin{proposition}\label{prop: MMD and rank r extended weak topology}
  Let $\Ban$ be a separable Banach space, $ (\bX^n)_{n\geq 0} \subset \APr(\Ban) $, $\bX \in \APr(\Ban)$, and $r \ge 0$.
	\begin{enumerate}
  \item \label{itm:1} If $ (\bX^n) $ converges to $ \bX $ in $(\APr(\Ban), \hat \tau_r)$, then $ d_{r}(\bX^n,\bX) \to 0 $ as $n \to \infty$.
		\item \label{itm:2} If $(\bX^n)_{n\geq 0}$ is contained in a compact set of $(\APr(\Ban), \hat \tau_r)$, then $ d_{r}(\bX^n,\bX) \to 0 $ as $n \to \infty $ implies that $ (\bX^n) $ converges to $ \bX $ in $(\APr(\Ban), \hat \tau_r)$.
	\end{enumerate}
\end{proposition}
\begin{proof}
(\ref{itm:1}) 
 $\bX^k$ converging to $\bX$ in the rank $r$ extended weak topology means that $\Law{\hat \bX^{k,r}}$ converges to $\Law{\hat \bX^r}$. By Proposition \ref{prop: concrete description of expected signature of higher rank}, $\bar \EE[\bX^{r+1}_0] = \EE[\Sigo \circ \ESig{r}{\hat \bX^r}]$, and by Proposition \ref{prop: rank r exp sig injective}, $\Sigo \circ \esig{r}$ is a continuous and bounded function on $ \cP_r(\Ban) $. The implication follows immediately.  \\
 (\ref{itm:2}) By assumption, $ (\bX^k)_{k\geq 0}$ is contained in a compact set with respect to the rank $r$ extended weak topology on $\APr(\vs)$. Hence, there exists a $\bY \in \APr(\vs)$ such that $\bX^k$ converges to $\bY$ in the rank $r$ extended weak topology. From the proof of (\ref{itm:2} $ \Rightarrow $ \ref{itm:1}) we have $ d_{r}(\bX^k,\bY) \to 0 $. Hence $d_{r}(\bX,\bY) = 0$, or equivalently, $\lVert \EE\bar \bX_0^{r+1} - \EE\bar \bY_0^{r+1} \rVert_{r+1} = 0$. Now using Theorem \ref{thm:2} we obtain that $\Law{\hat \bX^r} = \Law{\hat \bY^r}$.\\
 \end{proof}
We now relate $d_r$ and the rank $r$ extended weak topology with the adapted topology of rank $r$, $\tau_r$ (cf.~Definition \ref{def: adapted distribution and topology}).
\begin{proposition}\label{prop: extended r weak topology and adapted topology of rank r}
  For a separable Banach space $\Ban$, $\tau_r \subset \hat \tau_r$.
That is, convergence of $(\bX^n)$ in $(\APr(\Ban), \hat \tau_r)$ to $\bX$ implies convergence of $(\bX^n)$ to $\bX$ in $(\APr(\Ban), \tau_r)$. 	
Moreover, for processes evolving in a compact state space $\cK$, $\cK \subset \Ban$, the converse holds. That is
\begin{align}
  (\APr(\cK), \tau_r)= (\APr(\cK), \hat \tau_r). 
\end{align}

\end{proposition}
\begin{proof}
 First, it is easy to use an induction argument to show that for every $r \ge 0$, for every $f \in \cp_r$, there exists a bounded continuous function $\rho_f$ defined on $\idx \to \cP_r(\Ban)$ such that $f(\bX) = \rho_f(\hat \bX^r)$ for all $\bX \in \APr(\vs)$. As a consequence, if $ \bX^k $ converges to $ \bX $ in the rank $r$ extended weak topology; that is,  $\Law{\hat \bX^{k,r}}$ converges weakly to $\Law{\hat \bX^r}$ in $\cP(\idx \to \cP_r(\Ban))$, then it indeed holds that for any $f \in \cp_r$
$$
\lim_{k \to \infty} \EE[\rho_f(\hat \bX^{k,r})] = \EE[\rho_f(\hat \bX^r)]
$$
which is equivalent to $\lim_{k \to \infty} \EE[f(\hat \bX^{k,r})] = \EE[f(\hat \bX^r)]$; i.e., $ \bX^k $ converges to $ \bX $ in the adapted topology of rank $ r $. Also note that this result holds without assuming that $ (\bX^k)_{k\geq 0}$ is contained in a compact set with respect to the rank $r$ extended weak topology on $\APr(\vs)$.\\
On the other hand, by the definition of adapted functionals (cf. Definition \ref{def: adapted functionals}) one can easily verify that the class $\mathcal{A} := \{\rho_f : f \in \cp_r\}$ is a subalgebra in $C_b(\idx \to \cP_r(\cK); \R)$. Moreover, using the proof of \Implies{itm: equal mod r}{itm: equal pred proc2} in Theorem \ref{thm:2} we can also prove that $\mathcal{A}$ separate points on $(\idx \to \cP_r(\cK))$. Therefore, since the space $(\idx \to \cP_{r}(\cK))$ is obviously compact (recall that $\pr(\cK)$ is compact in the weak topology, and then by induction one can prove the compactness for all $\cP_r(\cK)$), $\mathcal{A}$ is dense in $C_b(K;\R)$ under the uniform topology by the Stone--Weierstrass theorem. Hence, for any given $\rho \in C_b(\idx \to \cP_{r}(\cK); \R)$ and any $\varepsilon > 0$, we can pick a $\rho_f \in \mathcal{A}$ such that $\sup_{x \in (\idx \to \cP_{r}(\cK)) }|\rho(x) - \rho_f(x)| \le \varepsilon$, and deduce that
\begin{align*}
|\EE[\rho(\hat \bX^{k,r})] - \EE[\rho(\hat \bX^r)]| &\le |\EE[\rho_f(\hat \bX^{k,r})] - \EE[\rho_f(\hat \bX^r)]| +2\varepsilon \\
&= |\EE[f(\hat \bX^{k})] - \EE[f(\hat \bX)]| +2\varepsilon  \to 0,
\end{align*}
where the last convergence holds as $\bX^k$ converges to $\bX$ in the adapted topology of rank $r$.
\end{proof}
Putting everything together, gives the following Theorem \ref{thm:main} which in turn implies Theorem~\ref{thm:mainA}. 

\begin{theorem}\label{thm:main}
  Let $\Ban$ be a separable Banach space and $r \ge 0$.
  Then for $\bX,\bY \in \APr(\Ban)$ 
  \begin{align}
    \left(\EE[f(\bX)] = \EE[f(\bY)] \quad \forall f \in \cp_r\right) 
    \text{ if and only if }{\Phi_r(\bX) = \Phi_r (\bY)}.
  \end{align}
Moreover, 
\begin{enumerate}
\item
  the map $\Phi_r$ locally induces the topology $\tau_r$,
\item
  the semimetric $d_r$ locally metrizes the topology $\tau_r$.
\end{enumerate}
If $\cK \subset \Ban$ is compact, then the above statements apply without localization, as stated in Theorem~\ref{thm:mainA}.
in this case of a compact state space, one can also replace the robust signature in the definition of $\Phi_r$ with the classical signature.  
\end{theorem}
Restricted to $r=0$, we recover the fact that the expected signature can (locally) metrize weak convergence \cite{Oberhauser18}; restricted to $r=1$ this (locally) induces Aldous extended weak topology \cite{Aldous81}. With $r=1$ and $ (\cF_t)_{0 \leq t \leq T} $ the natural filtration it adds another entry to the list in \cite{Beiglbock19} of distances that induce the same topology, and therefore we obtain Corollary \ref{corollary:In natural filtration case}.

\subsection{Tightness in Extended Weak topology: a Brief Discussion}\label{sec:tightness}
For the case $\Ban = \R^d$ one can obtain a tightness criterion for extended weak topology.
First we note that in this case $d_0$ really metrizes the usual weak convergence on $\R^d$. 
\begin{proposition}\label{prop: d0 metrizes weak convergence}
 $\mu_n$ converges to $\mu$ weakly in $\pr(\idx \to \R^d)$ if and only if $\lim_{n \to \infty} d_0(\mu_n,\mu) = 0$.
\end{proposition}
The proof of this proposition is given in Appendix \ref{app: signature, feature map and MMD}, see Proposition \ref{prop: normalized signature satisfies nice conditions} and Corollary \ref{cor: MMD characterizes weak convergence in appendix}.\\
Now we consider $r=1$; i.e, the Aldous' extended weak topology. Let $(\mu_t)_{t \in \idx}$ be a (discrete time) path in $(\idx \to \pr(\idx \to \R^d))$ such that $\mu_t \in \pr(\idx \to \R^d)$ for all $t \in \idx$. As before let $\Sigo$ denote a normalized signature map on $\idx \to \R^d$. Then associated with $(\mu_t)_{t \in \idx}$ we obtain a $\Ta{1}{\R^d}$--valued (discrete time) path $(\int \Sig{x} \mu_t(dx))_{t \in \idx}$. This mapping will be denoted by $F$, it is easy to see (by the Skorokhod representation theorem) that $F: (\idx \to \pr(\idx \to \R^d)) \to (\idx \to \Ta{1}{\R^d})$, $F((\mu_t)_{t \in \idx}) =(\int \Sig{x} \mu_t(dx))_{t \in \idx}$,  is continuous.
\begin{lemma}\label{lemma: pushforward keep tightness}
A set $\tops \subset \cP_2(\R^d)$ is tight if and only if the set
\[
\{F_{\sharp}\mu: \mu \in \tops \} \subset \cP_1( \Ta{1}{\R^d})
\]
is tight, where $F_{\sharp}$ means the pushforward operation.
\end{lemma}
\begin{proof}
Let $\text{Im}(F) \subset (\idx \to \Ta{1}{\R^d})$ be the image of $F$, which is endowed with the subspace topology inherited from the Hilbert space $\idx \to \Ta{1}{\R^d} = \Ta{1}{\R^d}^{\idx}$. By Proposition \ref{prop: d0 metrizes weak convergence} and the construction of $d_0$, we see that $F: (\idx \to \pr(\idx \to \R^d)) \to \text{Im}(F)$ is a homeomorphism. Hence the claim follows easily.
\end{proof}
Now let $\bX = ((\Omega,\PP, (\cF_t)_{t \in \idx}),X)\in \APr(\R^d)$, we know that $\hat \bX^1 \in (\idx \to \pr(\idx \to \R^d))$ and  $\Law{\hat \bX^1} \in \pr_2(\R^d)$. Note also that $F(\hat \bX^1) = (\EE[\Sig{X}|\cF_t])_{t\in \idx}$ is a $\Ta{1}{\R^d}$--valued martingale, and thus $F_{\sharp}(\Law{\hat \bX^1})$ is a martingale law on $\idx \to \Ta{1}{\R^d}$. Hence, by Lemma \ref{lemma: pushforward keep tightness} we obtain the following characterization of tightness set for the extended weak topology:
\begin{theorem}\label{thm: tightness in extended weak topology}
	Let $\tops \subset \APr(\R^d)$.
\begin{enumerate}
	\item $\tops$ is tight in the Aldous' extended weak topology if and only if the collection of martingale laws
	$$
	\{\Law{(\EE[\Sig{X}|\cF_t])_{t\in \idx}}: \bX \in \tops\}
	$$
	is tight in $\pr(\idx \to \Ta{1}{\R^d})$.
	\item If in addition that all filtrations are natural, then the following are equivalent
	\begin{itemize}
		\item $\tops$ is tight in the Aldous' extended weak topology, Hellwig's information topology, and all other topologies mentioned in Corollary \ref{corollary:In natural filtration case}.
		\item $\tops$ satisfies the Eder's conditions (\cite[Theorem 1.4]{eder2019compactness}).
		\item The collection of martingale laws
		$
		\{\Law{(\EE[\Sig{X}|\cF_t])_{t\in \idx}}: \bX \in \tops\}
		$
		is tight in $\cP_1(\Ta{1}{\R^d})$.
	\end{itemize}
\end{enumerate}
\end{theorem}
\begin{remark}\,
\begin{enumerate}
	\item The assumption that $\vs = \R^d$ cannot be removed because we need the local compactness of $(\idx \to \R^d)$ to ensure Proposition \ref{prop: d0 metrizes weak convergence}. This observation also prevent us from extending above results to higher rank extended weak topologies as, for example, the space $\pr(\idx \to \R^d)$ is not locally compact in general.
	\item
    Theorem~\ref{thm: tightness in extended weak topology} complements Eder's tightness theorem (\cite[Theorem 1.4]{eder2019compactness}).
    In particular, we  highlight that the expected signature map transforms the tightness of laws on a measure space into tightness of martingale laws on a Hilbert space.
    This allows to use tools from martingale theory and the Hilbert space to study the extended weak topology, e.g.~to define the Fourier transform (characteristic functions) for the law of prediction processes.
    Further, it suggests a concrete numerical way to check the tightness in extended weak topology by formulating it in the tensor algebra $\Ta{1}{\R^d}$; 
\end{enumerate}
\end{remark}
\section{Algorithms and Experiments} \label{sec: experiments}
In this Section we apply dynamic programming principles to derive algorithms that efficiently compute $\Phi_r(\bX)$ when the process $\bX$ is a Markov chain.
In the construction of $\Phi_r(\bX)$, the process $\bX$ is lifted to a process that evolves in the algebra $\Ta{r}{V}$ and dynamic programming naturally applies there as the following lemma shows.
\begin{lemma} \label{lem:markov}
	Let $A$ be an algebra and $\bZ \in \APr(A)$ be a Markov chain with finite support. 
	The function
	\begin{align}
	u_t(a) \coloneqq \EE[Z_{t+1}\cdots Z_T|Z_t=a ] 
	\end{align}
	satisfies for every $t=0,1,\ldots,T$ and $a \in A$ the recursion
	\begin{align}
	u_t(a) = \sum_{b \in A} b\PP(Z_{t+1} = b \vert Z_t = a) u_{t+1}(b)
	\end{align}
\end{lemma}
\begin{proof}
	This follows immediately from 
	\begin{align}
	u_t(a) = \EE[Z_{t+1} \cdots Z_T \vert Z_t = a] = \EE[Z_{t+1}u_{t+1}(Z_{t+1}) \vert Z_t = a].
	\end{align} 
\end{proof}
If $ \bX \in \APr(V)$ is a Markov chain, then the process $\bZ \in \APr(\Ta{1}{V})$ defined as 
\begin{align}
Z_t \coloneqq \exp({\Delta_t X}) \text{ with } \Delta_t X \coloneqq (X_{t+1}-X_t,1)
\end{align}
where $ \exp : V \to \Ta{1}{V} $ denotes the tensor exponential, is also a Markov chain that takes values in the algebra $ \Ta{1}{V} $. 
Recall that the signature $\Sig{x}$ of a path $ x : I \to \Ban $ is defined as 
\begin{align}
\Sig{x} = \exp({\Delta_0 x}) \exp({\Delta_1 x}) \cdots \exp({\Delta_T x}).\end{align}
Hence, Lemma~\ref{lem:markov} hints at an efficient way to compute $\Phi_0(\bX) \equiv \EE[\Sig{X}] = u_0(X_0) $ since the value function $u$ satisfies the recursion:
\begin{align} \label{eq:req1}
u_t(x)
=\sum_{y \in V}  \exp({ y-x }) \PP(X_{t+1} = y \vert X_t = x)u_{t+1}(y).
\end{align}
Algorithm \ref{alg:1} formulates this in pseudo-code by representing a Markov chain $\bX$ as tree: vertices are labelled by the attainable states $V$ of the Markov chain $\bX$; the process starts at time $t=0$ at a root vertex $r=X_0$; if the Markov chain at time $t$ has value $ a $ then we denote by \emph{a.children} the set of attainable states (vertices) at time $t+1$; the transition probability between two states $ a $ and $ b $ is denoted by $ p(a,b)$. 

\begin{algorithm}[H] \small
	\caption{Pseudo-code for $\Phi_0(\bX) $} \label{alg:1}
	\begin{algorithmic}[1] 
		\State {\bfseries Input:} A Markov chain $\bX$ represented as a rooted tree with root $r$ 
		\Procedure{$\operatorname{ExpSig_0}$}{$a$}
		\If{$ a.children $ {\bfseries is} empty}
		\State {\bfseries return} 1
		\EndIf
		\State sum $ \gets 0 $
		\For{$b$ {\bfseries in} $a.children$}
		\State sum $ \gets $ sum $ + p(a, b) \cdot \exp\big\{ b - a \big\} \cdot \operatorname{ExpSig}_0(b) $
		\EndFor
		\State {\bfseries return} sum
		\EndProcedure
		\State {\bfseries Output:} $\operatorname{ExpSig}_0(r)=\Phi_0(\bX)$
		
	\end{algorithmic}
\end{algorithm}

\begin{algorithm}[H] \small
	\caption{Pseudo-code for $\Phi_1(\bX) $} \label{alg:2}
	\begin{algorithmic}[1] 
		\State {\bfseries Input:} A Markov chain $\bX$ represented as a rooted tree with root $r$, $ s(a) $ denotes the signature of the sample path of $\bX$ that ends at $ a $.
		\Procedure{$\operatorname{ExpSig_1}$}{$a$}
		\State $a_1 \gets 0 $
		\State $a_2 \gets 0 $
		\For{$b$ {\bfseries in} $a.children$}
		\State $ b_1, b_2 \gets \operatorname{ExpSig}_1(b) $
		\State $a_1 \gets a_1 + p(a, b) * \exp\big\{ b - a \big\} * b_1 $
		\EndFor
		\For{$b$ {\bfseries in} $a.children$}
		\State $ b_1, b_2 \gets \operatorname{ExpSig}_1(b) $
		\State $a_2 \gets a_2 + p(a, b) * \exp\big\{ s({b}) * b_1 - s({a}) * a_1\big\} * b_2 $
		\EndFor
		\State {\bfseries return} $ a_1, a_2 $
		\EndProcedure
		\State {\bfseries Output:} $\operatorname{ExpSig}_1(r)=\Phi_1(\bX)$
	\end{algorithmic}
\end{algorithm}

\begin{lemma} \label{lem: complexity}
	Let $\bX \in \APr(U)$ be a Markov chain.
	If the rooted tree that represents $\bX$ has at most $ N+1 $ vertices and depth $ d $ then Algorithm~\ref{alg:1} computes $\Phi_1(\bX)$ with complexity  
	\begin{center}
		$ O\big(tN\big) $ in time and $ O\big(ds\big) $ in space,
	\end{center}
	where $ t,s $ are the time and space costs of computing and storing one call of $ \exp\big(\cdot\big) $ to the desired accuracy.
\end{lemma}
\begin{proof}
	Note that the bounds are clearly true if the tree has a a root with $ N $ children that are all leaves as the recursion will visit each child once and needs to store the return value as well as the execution stack of depth 1.
	Recursively, if the root has $ n $ children, each of which is the root of a sub-tree with $ n_i+1 $ vertices and depth $ d_i $ for $ i=1, \ldots, n $. Then the recursion visits each sub-tree once and adds the results of each sub-tree to the return value. The time and space complexities of the recursive call on the $ i $:th child are $ O(tn_i) $ and $ O(sd_i) $ respectively, hence the total time complexity is 
	\begin{align}
	O(\sum_{i=1}^n tn_i) = O(tN).
	\end{align}
	The space complexity is the maximum amount of space needed for the recursive call, plus the extra space for storing the value at the root and the execution stack, hence the total space complexity is 
	\begin{align}
	O( 1 + s + \max_{1 \leq i \leq n}(sd_i)) = O( 1 + s(d+1)) = O(ds),
	\end{align}
	proving the assertion.
\end{proof}

The computation of $\Phi_r(\bX)$ for $r > 1$ follows along the same lines, but since the notation gets increasingly cumbersome as $ r $ increases we only spell out the case $ r=2 $ in detail; the cases $r \ge 3$ follow analogous. 
We now apply Lemma~\ref{lem:markov} with $A \coloneqq \TaRank{2}(V)$ and $ Z_t = \exp({\Delta_t \bar{\bX}^1})  $ the function $ v^2_t(a) \coloneqq \EE[Z_{t+1}\cdots Z_T|Z_t=a ]  $ satisfies the recursion
\begin{align}
v_t^2(x) = \sum_{y \in V} \PP(X_{t+1} = y \vert X_t = x) \exp \Big\{ \EE[\Sig{X}\vert X_{t+1}=y]-\EE[\Sig{X}\vert X_{t}=x] \Big\} v_{t+1}^2(y).
\end{align}
This recursion is more involved as it requires two evaluations of $ \EE[\Sig{X}\vert X_{t}] $ at every step. However, if we assume that the process $ X $ is nowhere recombining we can rewrite this as the following system 
\begin{align} \label{eq:req2}
v_t^2(x) &= \sum_{y \in V} \PP(X_{t+1} = y \vert X_t = x)  \exp \big\{ \Sig{X \vert X_t = y}v^1_{t+1}(y)-\Sig{X \vert X_{t-1} = x}v^1_{t}(x) \big\} v^2_{t+1}(y) \\
v_t^1(x) &= \sum_{y\in V} \PP(X_{t+1} = y \vert X_t = x)  \exp \big\{ y-x \big\} v_{t+1}^1(y).
\end{align}
where $ \Sig{X \vert X_t = y} $ denotes the signature of the path $ X_0, \ldots, X_t $ such that $ X_t = y $, which is well defined since $ X $ is nowhere recombining by assumption. Note that $ v^1 $ is the same function as the one defined in Equation \ref{eq:req1}. Unfortunately $ v_t^2(x) $ depends on both $ v_t^1 $ and $ v_{t+1}^1 $ and since multiplication in $ \Ta{2}{V} $ is non-commutative there is no way to separate the two dependencies. Because of this $ v_t^1(x) $ needs to be computed before $ v_t^2(x) $ and the recursion is best solved using a Dynamic Programming approach, or by caching the relevant values at every function call. This approach is outlined in Algorithm \ref{alg:2}.

Algorithm \ref{alg:2} is more involved than Algorithm \ref{alg:1} as it requires the computation of $ \Sig{x} $ at every recursive call which has time complexity $ O(dt) $ where $ t $ is the time costs of computing one call of $ \exp_1\big(\big)$ to the desired accuracy, and $ d $ is the depth of $ x $. This can be remedied by memoising the values of $ \Sig{x} $ once computed, which brings the time complexity down to $ O(t) $ but takes up more space.

\begin{lemma}
	Let $ T,S $ be the time and space costs of computing and storing one call of $ \exp : \Ta{1}{V} \to \Ta{2}{V} $ to the desired accuracy, and $ t,s $ be the time and space costs of one call of $ \exp : V \to \Ta{1}{V} $. Assume that the tree has exactly $ N+1 $, depth $ d $ and maximal degree $ M $. Then the function ExpSig$ _2 $ can be implemented to be
	\begin{center}
		$ O\big((t+T)N\big) $ in time and $ O\big(d(MS + s)\big) $ in space.
	\end{center} 
\end{lemma}
\begin{proof}
	The same arguments made in the proof of Lemma \ref{lem: complexity} applies here too. If one caches $ \Sig{X \vert X_t = a} $ at every node one needs to store at most $ d $ values of $ \Sig{\cdot} $ and the cost of computing $ \exp \big\{ \Sig{X \vert X_t = b}v^1_{t+1}(b)-\Sig{X \vert X_{t-1} = a}v^1_{t}(a) \big\} $ is always $ O(T + t) $. By caching values of $ v_{t+1} $ in the first pass the maximum amount of space needed for each pass is $ MS $, since nodes are visited at most once the maximum amount of memory needed is $ dMS $.
\end{proof}

\begin{figure} 
	\begin{center}
		\includegraphics[scale=0.5]{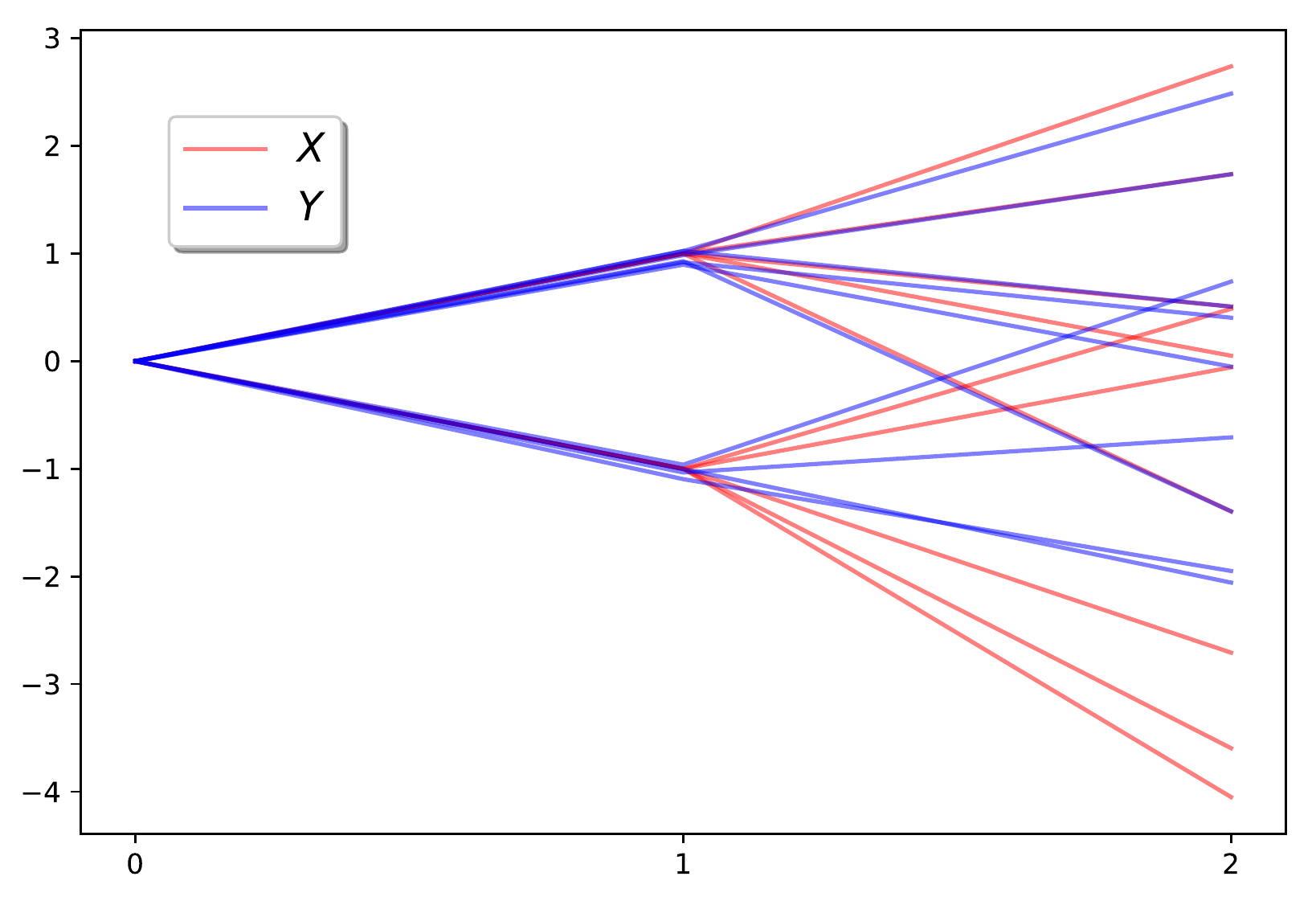}
	\end{center}
	\caption{The sample paths of $ \mu $ are plotted above for $ \varepsilon = 0.05 $. They are marked as either \textbf{X} in red, or \textbf{Y} in blue depending on if the sample comes from $ \bX^c $ or $ \bY^c $.}
	\label{fig:samples}
\end{figure}

\begin{remark}
	We assumed that $ \bX $ is a Markov chain and that $ \bX $ was nowhere recombining, that is that its value at time $ t $ uniquely determines $ X_1, \ldots, X_{t-1} $.
	We echo the usual remark that every process is Markovian by lifting it to a process in larger state space. Concretley, any process $\bX$ can be made to satisfy the assumptions of this section by considering instead the process $ Z_t = (X_1, \ldots, X_t) $.
\end{remark}
\begin{remark}[{Representing higher rank tensor algebras}]
	In order to implement any of the computations outlined above, one first needs to be able to represent the relevant algebras. $ \Ta{1}{V} $ is well known to be a graded algebra over $ V $, but $ \Ta{2}{V} $ has a more complicated multi-grading, and higher dimensional components which make computations trickier. See Appendix \ref{app:tensor algebras} for a more thorough discussion of how to write down the gradings, and the dimensions of the spaces involved, but note that it is always to write down (formal) gradings $ \Ta{r}{V} = \prod_{k\geq 0} \Ta{r}{V}_k $, where if $ V $ has dimension $ d $, then
	\begin{align}
	&\text{dim.}\Ta{1}{V}_k = (d+1)^k, \\
	&\text{dim.}\Ta{2}{V}_0 = 1, \quad \text{dim.}\Ta{2}{V}_1 = d+1, \\
	&\text{dim.}\Ta{2}{V}_k = (2d+3)\text{dim.}\Ta{2}{V}_{k-1} - (d+1)\text{dim.}\Ta{2}{V}_{k-2}.
	\end{align}
\end{remark}
\begin{figure} 
	\begin{center}
		\includegraphics[scale=0.5]{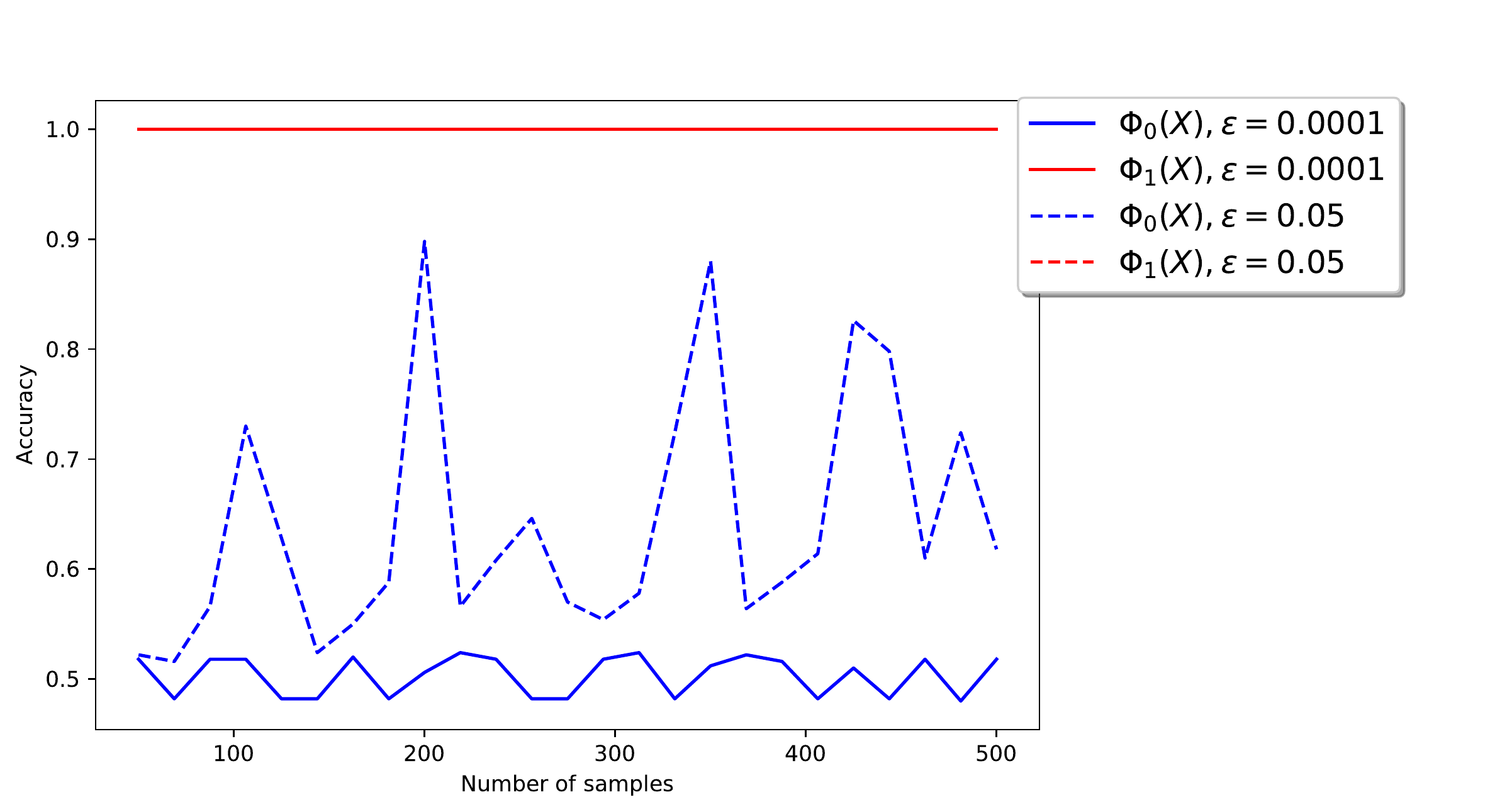}
	\end{center}
	\caption{The accuracies of the linear classifier trained on $ \Phi_0(\bX) $ and $ \Phi_1(\bX) $ is plotted against the number of samples used in blue and red respectively. Solid lines are used for $ \varepsilon = 10^{-4} $ and dashed lines for $ \varepsilon = 5 \times 10^{-2} $.}
	\label{fig:accs}
\end{figure}

\subsection{Experiment: Model Space and Linear Separability}
Expected signatures ($\Phi_0$ in our notation) are currently finding applications in machine learning.
One of their attractive properties is that they provide a hierarchical description of the law of a stochastic process; in the terminology of statistical learning the signature map is a so-called ``universal and characteristic'' feature map for paths, see Appendix~\ref{app: signature, feature map and MMD} .
However, expected signatures metrize weak convergence and hence completely ignore the filtration.

We now use the algorithms from the previous section to demonstrate on a simple numerical toy example that the geometry of the feature space of $\Phi_0$ is too simple in the sense that it fails to separate models with different filtrations. 
In contrast, the feature space of $\Phi_1$ is large enough to allow for a a linear separation.
\begin{example}[Mixtures of Adapted Processes]
	Define for every $c \in \R$ two processes $\bX^c, \bY^c \in \APr$ as 
	\begin{center}
		\begin{tabular}{ llll }
			$\bX^c$ : & $X_0 = 0$,  & $ X_1 = N_1, $ & $ X_2 = c+ N_2$,\\ 
			$\bY^c$ : & $ Y_0 = 0 $, & $ Y_1 = \sqrt{1-\varepsilon^2}M_1 + \varepsilon c , $ & $ Y_2 = c +M_2 $.
		\end{tabular}
	\end{center}
	where $M_1,M_2,N_1,N_2$ are pairwise independent Binomial random variables
	and $ \varepsilon > 0 $ is fixed. $ \bX^c $ and $ \bY^c $ are both equipped with their natural filtrations. Note that for a fixed $c$ and small $\varepsilon$, $\Law{\bX^c} \approx \Law{\bY^c}$, analogously to Example~\ref{ex:ost} resp.~Figure~\ref{fig:example}. 
	$\APr(\R)$ is equipped with a probability measure $\mu$ as follows: a sample from $\mu$ consists of sampling a $C\sim N(0,1)$ and then selecting with probability $0.5$ the process $X^C$ and with probability $0.5$ the process $Y^C$.
	Figure~\ref{fig:samples} shows for each sample from $\mu$ (an adapted process) one sample trajectory from this process. 
\end{example}
We ran the following experiment: We sampled $ 1000 $ processes from $\mu$ and labelled the processes corresponding to whether $\bX^c$ or $\bY^c$ was sampled. We then computed $ \Phi_0 $ and $ \Phi_1 $ for each sample truncated at level $ 6 $ and level $3$ respectively\footnote{This corresponds to 127 coordinates for $\Phi_0$ resp.~76 coordinates for $\Phi_1$, hence is in favour of $\Phi_0$.} and normalized the features. This was then split into a training set and test set -- both of size $ 500 $ -- and for $ 50 \leq m \leq 500$ a Support Vector Machine classifier \cite{svm98} was trained on $ m $ data points from the training set with $\Phi_0$ resp.~$\Phi_1$ as feature map. 

Figure \ref{fig:accs} shows the accuracies of the resulting classifiers on the test set. Observe that for small values of $ \varepsilon $, the classifier on $ \Phi_0 $ is essentially guessing, and even for larger values it does not converge well, the classifier on $ \Phi_1 $ converges immediately however, which is to be expected as $ \Phi_1 $ is able to separate $ \bX^c $ and $ \bY^c $ independently of the value of $ \varepsilon $.

We emphasize that although this is a toy example, it demonstrates how the expected signature can fail to pick up essential properties of a model and that higher rank expected signature provide additional features that linearise complex dependencies between law and filtration.

\begin{ackn}
	PB is supported by the Engineering and Physical Sciences Research Council [EP/R513295/1].
	CL is supported by the SNSF Grant [P2EZP2\_188068].
	HO is supported by the EPSRC grant ``Datasig'' [EP/S026347/1], the Alan Turing Institute, and the Oxford-Man Institute. 
	HO would like to thank Manu Eder for helpful discussions.
\end{ackn}

\small

\bibliography{references} 
\bibliographystyle{alpha}

\newpage
\normalsize
\appendix

\section{Details for Example \ref{ex:HK}} \label{app:A}

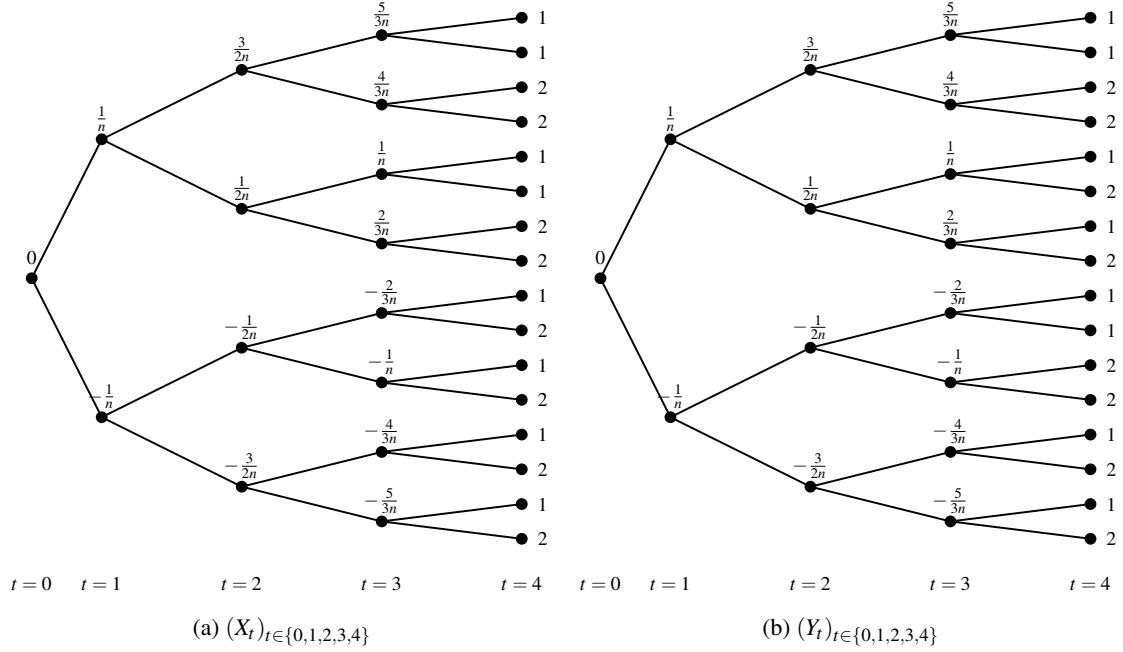
\begin{figure}[t!]
	\begin{subfigure}{.5\textwidth}
		\resizebox{\textwidth}{!}{%
			\begin{tikzpicture}[shorten >=1pt,draw=black!50]
			\node () at (3,-4.4) {\footnotesize$ t=0 $};
			\node () at (4,-4.4) {\footnotesize$ t=1 $};
			\node () at (6,-4.4) {\footnotesize$ t=2 $};
			\node () at (8,-4.4) {\footnotesize$ t=3 $};
			\node () at (10,-4.4) {\footnotesize$ t=4 $};

			\draw[black, fill=black, thick] (3,0) circle (2pt);
			\node () at (3,.3) {\footnotesize$ 0 $};
			
			\draw[black, fill=black, thick] (4,2) circle (2pt);
			\node () at (4,2.3) {\footnotesize$ \frac1n $};
			\draw[black, fill=black, thick] (4,-2) circle (2pt);
			\node () at (4,-1.7) {\footnotesize$ -\frac1n $};
			\draw[black, thick] (3,0) -- (4,2);
			\draw[black, thick] (3,0) -- (4,-2);
			
			\draw[black, fill=black, thick] (6,3) circle (2pt);
			\node () at (6,3.3) {\footnotesize$ \frac3{2n} $};
			\draw[black, fill=black, thick] (6,-3) circle (2pt);
			\node () at (6,-2.7) {\footnotesize$ -\frac3{2n} $};
			\draw[black, thick] (4,2) -- (6,3);
			\draw[black, thick] (4,-2) -- (6,-3);
			
			\draw[black, fill=black, thick] (6,1) circle (2pt);
			\node () at (6,1.3) {\footnotesize$ \frac1{2n} $};
			\draw[black, fill=black, thick] (6,-1) circle (2pt);
			\node () at (6,-0.7) {\footnotesize$ -\frac1{2n} $};
			\draw[black, thick] (4,2) -- (6,1);
			\draw[black, thick] (4,-2) -- (6,-1);
			
			\draw[black, fill=black, thick] (8,3.5) circle (2pt);
			\node () at (8,3.8) {\footnotesize$ \frac5{3n} $};
			\draw[black, fill=black, thick] (8,-3.5) circle (2pt);
			\node () at (8,-3.2) {\footnotesize$ -\frac5{3n} $};
			\draw[black, thick] (6,3) -- (8,3.5);
			\draw[black, thick] (6,-3) -- (8,-3.5);
			
			\draw[black, fill=black, thick] (8,2.5) circle (2pt);
			\node () at (8,2.8) {\footnotesize$ \frac4{3n} $};
			\draw[black, fill=black, thick] (8,-2.5) circle (2pt);
			\node () at (8,-2.2) {\footnotesize$ -\frac4{3n} $};
			\draw[black, thick] (6,3) -- (8,2.5);
			\draw[black, thick] (6,-3) -- (8,-2.5);
			
			\draw[black, fill=black, thick] (8,1.5) circle (2pt);
			\node () at (8,1.8) {\footnotesize$ \frac1{n} $};
			\draw[black, fill=black, thick] (8,-1.5) circle (2pt);
			\node () at (8,-1.2) {\footnotesize$ -\frac1{n} $};
			\draw[black, thick] (6,1) -- (8,1.5);
			\draw[black, thick] (6,-1) -- (8,-1.5);
			
			\draw[black, fill=black, thick] (8,0.5) circle (2pt);
			\node () at (8,0.8) {\footnotesize$ \frac2{3n} $};
			\draw[black, fill=black, thick] (8,-0.5) circle (2pt);
			\node () at (8,-.2) {\footnotesize$ -\frac2{3n} $};
			\draw[black, thick] (6,1) -- (8,0.5);
			\draw[black, thick] (6,-1) -- (8,-0.5);
			
			\draw[black, fill=black, thick] (10,3.75) circle (2pt);
			\node () at (10.3,3.75) {\footnotesize$1$};
			\draw[black, fill=black, thick] (10,3.25) circle (2pt);
			\node () at (10.3,3.25) {\footnotesize$1$};
			\draw[black, thick] (8,3.5) -- (10,3.75);
			\draw[black, thick] (8,3.5) -- (10,3.25);
			
			\draw[black, fill=black, thick] (10,2.75) circle (2pt);
			\node () at (10.3,2.75) {\footnotesize$2$};
			\draw[black, fill=black, thick] (10,2.25) circle (2pt);
			\node () at (10.3,2.25) {\footnotesize$2$};
			\draw[black, thick] (8,2.5) -- (10,2.75);
			\draw[black, thick] (8,2.5) -- (10,2.25);
			
			\draw[black, fill=black, thick] (10,1.75) circle (2pt);
			\node () at (10.3,1.75) {\footnotesize$1$};
			\draw[black, fill=black, thick] (10,1.25) circle (2pt);
			\node () at (10.3,1.25) {\footnotesize$1$};
			\draw[black, thick] (8,1.5) -- (10,1.75);
			\draw[black, thick] (8,1.5) -- (10,1.25);
			
			\draw[black, fill=black, thick] (10,0.75) circle (2pt);
			\node () at (10.3,0.75) {\footnotesize$2$};
			\draw[black, fill=black, thick] (10,0.25) circle (2pt);
			\node () at (10.3,0.25) {\footnotesize$2$};
			\draw[black, thick] (8,0.5) -- (10,0.75);
			\draw[black, thick] (8,0.5) -- (10,0.25);
			
			\draw[black, fill=black, thick] (10,-3.75) circle (2pt);
			\node () at (10.3,-3.75) {\footnotesize$2$};
			\draw[black, fill=black, thick] (10,-3.25) circle (2pt);
			\node () at (10.3,-3.25) {\footnotesize$1$};
			\draw[black, thick] (8,-3.5) -- (10,-3.75);
			\draw[black, thick] (8,-3.5) -- (10,-3.25);
			
			\draw[black, fill=black, thick] (10,-2.75) circle (2pt);
			\node () at (10.3,-2.75) {\footnotesize$2$};
			\draw[black, fill=black, thick] (10,-2.25) circle (2pt);
			\node () at (10.3,-2.25) {\footnotesize$1$};
			\draw[black, thick] (8,-2.5) -- (10,-2.75);
			\draw[black, thick] (8,-2.5) -- (10,-2.25);
			
			\draw[black, fill=black, thick] (10,-1.75) circle (2pt);
			\node () at (10.3,-1.75) {\footnotesize$2$};
			\draw[black, fill=black, thick] (10,-1.25) circle (2pt);
			\node () at (10.3,-1.25) {\footnotesize$1$};
			\draw[black, thick] (8,-1.5) -- (10,-1.75);
			\draw[black, thick] (8,-1.5) -- (10,-1.25);
			
			\draw[black, fill=black, thick] (10,-0.75) circle (2pt);
			\node () at (10.3,-0.75) {\footnotesize$2$};
			\draw[black, fill=black, thick] (10,-0.25) circle (2pt);
			\node () at (10.3,-0.25) {\footnotesize$1$};
			\draw[black, thick] (8,-0.5) -- (10,-0.75);
			\draw[black, thick] (8,-0.5) -- (10,-0.25);
			\end{tikzpicture}
		}
		\caption{$ (X_t)_{t\in \{0, 1, 2, 3, 4\}} $}
	\end{subfigure}%
	~
	\begin{subfigure}{.5\textwidth}
		\resizebox{\textwidth}{!}{%
			\begin{tikzpicture}[shorten >=1pt,draw=black!50]
			\node () at (3,-4.4) {\footnotesize$ t=0 $};
			\node () at (4,-4.4) {\footnotesize$ t=1 $};
			\node () at (6,-4.4) {\footnotesize$ t=2 $};
			\node () at (8,-4.4) {\footnotesize$ t=3 $};
			\node () at (10,-4.4) {\footnotesize$ t=4 $};
			
			\draw[black, fill=black, thick] (3,0) circle (2pt);
			\node () at (3,.3) {\footnotesize$ 0 $};
			
			\draw[black, fill=black, thick] (4,2) circle (2pt);
			\node () at (4,2.3) {\footnotesize$ \frac1n $};
			\draw[black, fill=black, thick] (4,-2) circle (2pt);
			\node () at (4,-1.7) {\footnotesize$ -\frac1n $};
			\draw[black, thick] (3,0) -- (4,2);
			\draw[black, thick] (3,0) -- (4,-2);
			
			\draw[black, fill=black, thick] (6,3) circle (2pt);
			\node () at (6,3.3) {\footnotesize$ \frac3{2n} $};
			\draw[black, fill=black, thick] (6,-3) circle (2pt);
			\node () at (6,-2.7) {\footnotesize$ -\frac3{2n} $};
			\draw[black, thick] (4,2) -- (6,3);
			\draw[black, thick] (4,-2) -- (6,-3);
			
			\draw[black, fill=black, thick] (6,1) circle (2pt);
			\node () at (6,1.3) {\footnotesize$ \frac1{2n} $};
			\draw[black, fill=black, thick] (6,-1) circle (2pt);
			\node () at (6,-0.7) {\footnotesize$ -\frac1{2n} $};
			\draw[black, thick] (4,2) -- (6,1);
			\draw[black, thick] (4,-2) -- (6,-1);
			
			\draw[black, fill=black, thick] (8,3.5) circle (2pt);
			\node () at (8,3.8) {\footnotesize$ \frac5{3n} $};
			\draw[black, fill=black, thick] (8,-3.5) circle (2pt);
			\node () at (8,-3.2) {\footnotesize$ -\frac5{3n} $};
			\draw[black, thick] (6,3) -- (8,3.5);
			\draw[black, thick] (6,-3) -- (8,-3.5);
			
			\draw[black, fill=black, thick] (8,2.5) circle (2pt);
			\node () at (8,2.8) {\footnotesize$ \frac4{3n} $};
			\draw[black, fill=black, thick] (8,-2.5) circle (2pt);
			\node () at (8,-2.2) {\footnotesize$ -\frac4{3n} $};
			\draw[black, thick] (6,3) -- (8,2.5);
			\draw[black, thick] (6,-3) -- (8,-2.5);
			
			\draw[black, fill=black, thick] (8,1.5) circle (2pt);
			\node () at (8,1.8) {\footnotesize$ \frac1n $};
			\draw[black, fill=black, thick] (8,-1.5) circle (2pt);
			\node () at (8,-1.2) {\footnotesize$ -\frac1n $};
			\draw[black, thick] (6,1) -- (8,1.5);
			\draw[black, thick] (6,-1) -- (8,-1.5);
			
			\draw[black, fill=black, thick] (8,0.5) circle (2pt);
			\node () at (8,0.8) {\footnotesize$ \frac2{3n} $};
			\draw[black, fill=black, thick] (8,-0.5) circle (2pt);
			\node () at (8,-.2) {\footnotesize$ -\frac2{3n} $};
			\draw[black, thick] (6,1) -- (8,0.5);
			\draw[black, thick] (6,-1) -- (8,-0.5);
			
			\draw[black, fill=black, thick] (10,3.75) circle (2pt);
			\node () at (10.3,3.75) {\footnotesize$1$};
			\draw[black, fill=black, thick] (10,3.25) circle (2pt);
			\node () at (10.3,3.25) {\footnotesize$1$};
			\draw[black, thick] (8,3.5) -- (10,3.75);
			\draw[black, thick] (8,3.5) -- (10,3.25);
			
			\draw[black, fill=black, thick] (10,2.75) circle (2pt);
			\node () at (10.3,2.75) {\footnotesize$2$};
			\draw[black, fill=black, thick] (10,2.25) circle (2pt);
			\node () at (10.3,2.25) {\footnotesize$2$};
			\draw[black, thick] (8,2.5) -- (10,2.75);
			\draw[black, thick] (8,2.5) -- (10,2.25);
			
			\draw[black, fill=black, thick] (10,1.75) circle (2pt);
			\node () at (10.3,1.75) {\footnotesize$1$};
			\draw[black, fill=black, thick] (10,1.25) circle (2pt);
			\node () at (10.3,1.25) {\footnotesize$2$};
			\draw[black, thick] (8,1.5) -- (10,1.75);
			\draw[black, thick] (8,1.5) -- (10,1.25);
			
			\draw[black, fill=black, thick] (10,0.75) circle (2pt);
			\node () at (10.3,0.75) {\footnotesize$1$};
			\draw[black, fill=black, thick] (10,0.25) circle (2pt);
			\node () at (10.3,0.25) {\footnotesize$2$};
			\draw[black, thick] (8,0.5) -- (10,0.75);
			\draw[black, thick] (8,0.5) -- (10,0.25);
			
			\draw[black, fill=black, thick] (10,-3.75) circle (2pt);
			\node () at (10.3,-3.75) {\footnotesize$2$};
			\draw[black, fill=black, thick] (10,-3.25) circle (2pt);
			\node () at (10.3,-3.25) {\footnotesize$1$};
			\draw[black, thick] (8,-3.5) -- (10,-3.75);
			\draw[black, thick] (8,-3.5) -- (10,-3.25);
			
			\draw[black, fill=black, thick] (10,-2.75) circle (2pt);
			\node () at (10.3,-2.75) {\footnotesize$2$};
			\draw[black, fill=black, thick] (10,-2.25) circle (2pt);
			\node () at (10.3,-2.25) {\footnotesize$1$};
			\draw[black, thick] (8,-2.5) -- (10,-2.75);
			\draw[black, thick] (8,-2.5) -- (10,-2.25);
			
			\draw[black, fill=black, thick] (10,-1.75) circle (2pt);
			\node () at (10.3,-1.75) {\footnotesize$2$};
			\draw[black, fill=black, thick] (10,-1.25) circle (2pt);
			\node () at (10.3,-1.25) {\footnotesize$2$};
			\draw[black, thick] (8,-1.5) -- (10,-1.75);
			\draw[black, thick] (8,-1.5) -- (10,-1.25);
			
			\draw[black, fill=black, thick] (10,-0.75) circle (2pt);
			\node () at (10.3,-0.75) {\footnotesize$1$};
			\draw[black, fill=black, thick] (10,-0.25) circle (2pt);
			\node () at (10.3,-0.25) {\footnotesize$1$};
			\draw[black, thick] (8,-0.5) -- (10,-0.75);
			\draw[black, thick] (8,-0.5) -- (10,-0.25);
			\end{tikzpicture}
		}
		\caption{$ (Y_t)_{t\in \{0, 1, 2, 3, 4\}} $}
	\end{subfigure}%
	\caption{Two processes $ X $ and $ Y $ that converge weakly to the same process and such that the difference between their prediction processes converges weakly to zero, but does not converge in the rank $ 2 $ adapted topology. Before $ t=4 $ they move very little, with steps of order $ 1/n $ and at $ t=4 $ they jump to either $ 1 $ or $ 2 $ with equal probability. As $ n \to \infty $ they both converge weakly to the process that stay at $ 0 $ until $ t=4 $ when it jumps to either $ 1 $ or $ 2 $.}
	\label{fig:hk1}
\end{figure}

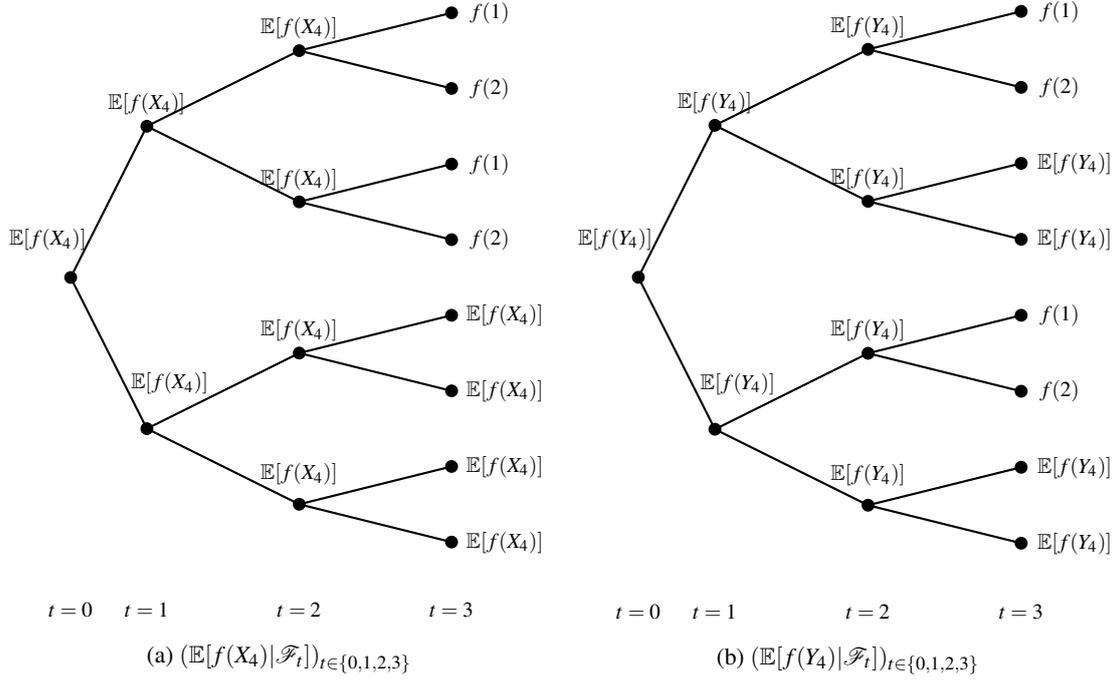
\begin{figure}
	\begin{subfigure}{.5\textwidth}
		\resizebox{\textwidth}{!}{%
			\begin{tikzpicture}[shorten >=1pt,draw=black!50]
			\node () at (3,-4.4) {\footnotesize$ t=0 $};
			\node () at (4,-4.4) {\footnotesize$ t=1 $};
			\node () at (6,-4.4) {\footnotesize$ t=2 $};
			\node () at (8,-4.4) {\footnotesize$ t=3 $};			
			
			\draw[black, fill=black, thick] (3,0) circle (2pt);
			\node () at (2.7,.5) {\footnotesize$ \EE[f(X_4)] $};
			
			\draw[black, fill=black, thick] (4,2) circle (2pt);
			\node () at (4,2.3) {\footnotesize$ \EE[f(X_4)] $};
			\draw[black, fill=black, thick] (4,-2) circle (2pt);
			\node () at (4.3,-1.4) {\footnotesize$ \EE[f(X_4)] $};
			\draw[black, thick] (3,0) -- (4,2);
			\draw[black, thick] (3,0) -- (4,-2);
			
			\draw[black, fill=black, thick] (6,3) circle (2pt);
			\node () at (6,3.3) {\footnotesize$ \EE[f(X_4)] $};
			\draw[black, fill=black, thick] (6,-3) circle (2pt);
			\node () at (6,-2.6) {\footnotesize$ \EE[f(X_4)] $};
			\draw[black, thick] (4,2) -- (6,3);
			\draw[black, thick] (4,-2) -- (6,-3);
			
			\draw[black, fill=black, thick] (6,1) circle (2pt);
			\node () at (6,1.3) {\footnotesize$ \EE[f(X_4)] $};
			\draw[black, fill=black, thick] (6,-1) circle (2pt);
			\node () at (6,-0.7) {\footnotesize$ \EE[f(X_4)] $};
			\draw[black, thick] (4,2) -- (6,1);
			\draw[black, thick] (4,-2) -- (6,-1);
			
			\draw[black, fill=black, thick] (8,3.5) circle (2pt);
			\node () at (8.5,3.5) {\footnotesize$ f(1) $};
			\draw[black, fill=black, thick] (8,-3.5) circle (2pt);
			\node () at (8.5,2.5) {\footnotesize$ f(2) $};
			\draw[black, thick] (6,3) -- (8,3.5);
			\draw[black, thick] (6,-3) -- (8,-3.5);
			
			\draw[black, fill=black, thick] (8,2.5) circle (2pt);
			\node () at (8.5,1.5) {\footnotesize$ f(1) $};
			\draw[black, fill=black, thick] (8,-2.5) circle (2pt);
			\node () at (8.5,0.5) {\footnotesize$ f(2) $};
			\draw[black, thick] (6,3) -- (8,2.5);
			\draw[black, thick] (6,-3) -- (8,-2.5);
			
			\draw[black, fill=black, thick] (8,1.5) circle (2pt);
			\node () at (8.7,-0.5) {\footnotesize$ \EE[f(X_4)] $};
			\draw[black, fill=black, thick] (8,-1.5) circle (2pt);
			\node () at (8.7,-1.5) {\footnotesize$ \EE[f(X_4)] $};
			\draw[black, thick] (6,1) -- (8,1.5);
			\draw[black, thick] (6,-1) -- (8,-1.5);
			
			\draw[black, fill=black, thick] (8,0.5) circle (2pt);
			\node () at (8.7,-2.5) {\footnotesize$ \EE[f(X_4)] $};
			\draw[black, fill=black, thick] (8,-0.5) circle (2pt);
			\node () at (8.7,-3.5) {\footnotesize$ \EE[f(X_4)] $};
			\draw[black, thick] (6,1) -- (8,0.5);
			\draw[black, thick] (6,-1) -- (8,-0.5);
			\end{tikzpicture}
		}
		\caption{$ (\EE[f(X_4)\vert \cF_t])_{t\in \{0, 1, 2, 3\} } $}
	\end{subfigure}%
	~
	\begin{subfigure}{.5\textwidth}
		\resizebox{\textwidth}{!}{%
			\begin{tikzpicture}[shorten >=1pt,draw=black!50]
			\node () at (3,-4.4) {\footnotesize$ t=0 $};
			\node () at (4,-4.4) {\footnotesize$ t=1 $};
			\node () at (6,-4.4) {\footnotesize$ t=2 $};
			\node () at (8,-4.4) {\footnotesize$ t=3 $};			
			
			\draw[black, fill=black, thick] (3,0) circle (2pt);
			\node () at (2.7,.5) {\footnotesize$ \EE[f(Y_4)] $};
			
			\draw[black, fill=black, thick] (4,2) circle (2pt);
			\node () at (4,2.3) {\footnotesize$ \EE[f(Y_4)] $};
			\draw[black, fill=black, thick] (4,-2) circle (2pt);
			\node () at (4.3,-1.4) {\footnotesize$ \EE[f(Y_4)] $};
			\draw[black, thick] (3,0) -- (4,2);
			\draw[black, thick] (3,0) -- (4,-2);
			
			\draw[black, fill=black, thick] (6,3) circle (2pt);
			\node () at (6,3.3) {\footnotesize$ \EE[f(Y_4)] $};
			\draw[black, fill=black, thick] (6,-3) circle (2pt);
			\node () at (6,-2.6) {\footnotesize$ \EE[f(Y_4)] $};
			\draw[black, thick] (4,2) -- (6,3);
			\draw[black, thick] (4,-2) -- (6,-3);
			
			\draw[black, fill=black, thick] (6,1) circle (2pt);
			\node () at (6,1.3) {\footnotesize$ \EE[f(Y_4)] $};
			\draw[black, fill=black, thick] (6,-1) circle (2pt);
			\node () at (6,-0.7) {\footnotesize$ \EE[f(Y_4)] $};
			\draw[black, thick] (4,2) -- (6,1);
			\draw[black, thick] (4,-2) -- (6,-1);
			
			\draw[black, fill=black, thick] (8,3.5) circle (2pt);
			\node () at (8.5,3.5) {\footnotesize$ f(1) $};
			\draw[black, fill=black, thick] (8,-3.5) circle (2pt);
			\node () at (8.5,2.5) {\footnotesize$ f(2) $};
			\draw[black, thick] (6,3) -- (8,3.5);
			\draw[black, thick] (6,-3) -- (8,-3.5);
			
			\draw[black, fill=black, thick] (8,2.5) circle (2pt);
			\node () at (8.7,1.5) {\footnotesize$ \EE[f(Y_4)] $};
			\draw[black, fill=black, thick] (8,-2.5) circle (2pt);
			\node () at (8.7,0.5) {\footnotesize$ \EE[f(Y_4)] $};
			\draw[black, thick] (6,3) -- (8,2.5);
			\draw[black, thick] (6,-3) -- (8,-2.5);
			
			\draw[black, fill=black, thick] (8,1.5) circle (2pt);
			\node () at (8.5,-0.5) {\footnotesize$ f(1) $};
			\draw[black, fill=black, thick] (8,-1.5) circle (2pt);
			\node () at (8.5,-1.5) {\footnotesize$ f(2) $};
			\draw[black, thick] (6,1) -- (8,1.5);
			\draw[black, thick] (6,-1) -- (8,-1.5);
			
			\draw[black, fill=black, thick] (8,0.5) circle (2pt);
			\node () at (8.7,-2.5) {\footnotesize$ \EE[f(Y_4)] $};
			\draw[black, fill=black, thick] (8,-0.5) circle (2pt);
			\node () at (8.7,-3.5) {\footnotesize$ \EE[f(Y_4)] $};
			\draw[black, thick] (6,1) -- (8,0.5);
			\draw[black, thick] (6,-1) -- (8,-0.5);
			\end{tikzpicture}
		}
		\caption{$ (\EE[f(Y_4)\vert \cF_t])_{t\in \{0, 1, 2, 3\} } $}
	\end{subfigure}
	\caption{For any fixed $ f \in C_b(\R) $, the processes $ t \mapsto \EE[f(X_4)\vert \cF_t] $ and $ t \mapsto \EE[f(Y_4)\vert \cF_t] $ have the same distribution, so $ \hat X - \hat Y $ goes to $ 0 $ as $ n \to \infty $. The rank $ 2 $ prediction processes are not the same however}
	\label{fig:hk2}
\end{figure}

	Consider the Probability space $ \Omega = \{ 1, \ldots, 16 \} $ equipped with the counting measure and the filtration 
	\begin{align}
	\cF_0 &= \{ \Omega, \varnothing \}, \\
	\cF_1 &= \sigma\langle \{ 1, \ldots, 8 \}, \{ 9, \ldots, 16 \} \rangle, \\ 
	\cF_2 &= \sigma\langle \{ 1, 2, 3, 4 \}, \{ 5, 6, 7, 8 \}, \{ 9, 10, 11, 12 \}, \{ 13, 14, 15, 16 \} \rangle, \\ 
	\cF_3 &= \sigma\langle \{ 1, 2 \}, \{ 3, 4 \}, \{ 5, 6 \}, \{ 7, 8 \}, \{ 9, 10 \}, \{ 11, 12 \} \{ 13, 14 \}, \{ 15, 16 \} \rangle, \\
	\cF_4 &= 2^\Omega.
	\end{align}
	Define the two processes 
	\begin{align}
	X_0 = X_1 = X_2 = X_3 = 0, 
	X_4 : \begin{cases}
	1,2,5,6,9,11,13,15 \mapsto 1 \\
	3,4,7,8,10,12,14,16 \mapsto 2,
	\end{cases} \\
	Y_0 = Y_1 = Y_2 = Y_3 = 0, 
	Y_4 : \begin{cases}
	1,2,5,7,9,10,13,15 \mapsto 1 \\
	3,4,6,8,11,12,14,16 \mapsto 2,
	\end{cases}.
	\end{align}
	If the above construction looks unnatural the reader is also invited to think of the filtration as being the natural filtration associated to the processes and that instead of staying at $ 0 $ until time $ 4 $, they move with step size of order $ 1/n $ in such a way to generate $ \cF $, as in Figure \ref{fig:hk1} and \ref{fig:hk2}.
	Clearly the image measure of $ X $ and $ Y $ are the same, so $ \EE f(X) = \EE f(Y) $ for any $ f \in \R^{\{0,1,2\}} $. Moreover:
	\begin{align}
	&\EE[f(X_4)\vert \cF_0] = \EE[f(X_4)\vert \cF_1] = \EE[f(X_4)\vert \cF_2] = \frac12\big(f(1)+f(2)\big), \\
	&\EE[f(X_4)\vert \cF_3] : \begin{cases}
	\{1,2\},\{5,6\} \mapsto f(1) \\
	\{3,4\},\{7,8\} \mapsto f(2), \\
	\{9, 10\}, \{11, 12\}, \{13, 14\}, \{15, 16\} \mapsto \frac12\big(f(1)+f(2)\big),
	\end{cases} \\
	&\EE[f(Y_4)\vert \cF_0] = \EE[f(Y_4)\vert \cF_1] = \EE[f(Y_4)\vert \cF_2] = \frac12\big(f(1)+f(2)\big), \\
	&\EE[f(Y_4)\vert \cF_3] : \begin{cases}
	\{1,2\},\{9,10\} \mapsto f(1) \\
	\{3,4\},\{11,12\} \mapsto f(2), \\
	\{5, 6\}, \{7, 8\}, \{13, 14\}, \{15, 16\} \mapsto \frac12\big(f(1)+f(2)\big),
	\end{cases} 
	\end{align}
	since the image measure of the above processes are the same, $ \EE[g(\EE[f(X)\vert \cF])] = \EE[g(\EE[f(Y)\vert \cF])] $ for any $ f,g \in \R^{\{0,1,2\}} $ and therefore they have the same prediction process. However, it can bee seen that
	\begin{align}
	&\EE[\EE[X_4\vert \cF_3]^2\vert \cF_1] : \begin{cases}
	\{1,2,3,4,5,6,7,8\} \mapsto \frac52 \\
	\{9,10,11,12,13,14,15,16\} \mapsto \frac94,
	\end{cases} \\
	&\EE[\EE[Y_4\vert \cF_3]^2\vert \cF_1] = \frac{19}8.
	\end{align}
	Hence the information structure in these processes are different, but this can't be seen by their prediction processes alone.

\section{Higher rank tensor algebras and their norms}\label{app:tensor algebras}
If $\vs$ is a Banach space with norm $ \lVert \cdot \rVert $, then we want to equip $\vs^{\otimes m}$ with a norm for every $m\ge 1$. In the general case some care is needed and we assume that all norms on tensor products are \emph{admissible} as defined below.
\begin{definition}\label{def: admissible norm}
	We say that $\| \cdot \|$ is an admissible norm on $(\vs^{\otimes m})_{m\geq 1}$, if: 
	\begin{enumerate}
		\item For any permutation $ \sigma : \{ 1, \ldots, n \} \to \{ 1, \ldots, n \} $
		\begin{align}
		\lVert v_1 \otimes \cdots \otimes v_n \rVert = \lVert v_{\sigma(1)} \otimes \cdots \otimes v_{\sigma(n)} \rVert.
		\end{align}
		\item For $ v \in V^{\otimes n}, w \in V^{\otimes m} $ it holds that
		\begin{align}
		\lVert v \otimes w \rVert \leq \lVert v \rVert \cdot \lVert w \rVert.
		\end{align}
	\end{enumerate}
	Both projective, and injective norms are admissible. See \cite{Ryan02} for more details.
\end{definition}
Recall that if $ V $ is a vector space, then $ \mathbf{ut}(V) := \bigoplus_{m\geq 0} \vs^{\otimes m} $ is the tensor algebra over $ V $, and $ \mathbf{ut}(V)-1 := \bigoplus_{m\geq 1} \vs^{\otimes m} $ is the non-unital tensor algebra over $ V $. The higher order tensor algebras are defined inductively as follows:
\begin{definition}
	Let $ \vs $ be a normed space. Define the spaces
	\begin{align}
	\ta{0}{\vs} &= \vs, \quad 
	\ta{r}{\vs} = \bigoplus_{m\geq 0} \big( \ta{r-1}{\vs} -1 \big)^{\otimes m}, \quad &r\geq 1. \\
	\ata{0}{\vs} &= \vs, \quad 
	\ata{r}{\vs} = \bigoplus_{m\geq 0} \big( \R\oplus \ata{r-1}{\vs} - 1\big)^{\otimes m}, \quad &r\geq 1.
	\end{align}
\end{definition}
A module $ A $ is said to be multi-graded if there is a monoid $ M $ such that
\begin{align}
A = \bigoplus_{m\in M} A_m
\end{align}
and $ A_mA_n \subseteq A_{mn} $. In order to describe the multi-grading of $ \ta{r}{\vs} $ and $ \ata{r}{\vs} $ we will use the following lemma. We use the notation $ \cF\big[\cdot \big] $ for the free algebra generated by $ \cdot $ and $ \cM\big[\cdot \big] $ for the free monoid generated by $ \cdot $. The following follows from the definitions of $ \cF $ and $ \cM $ and is recorded here as a lemma.
\begin{lemma} \label{lem:MG}
	Let $ A_1, \ldots, A_n $ be multi-graded modules with respective multi-gradings $ M_1, \ldots, M_n $. Then $ \cF\big[ A_1, \ldots, A_n \big] $ is multi-graded by $ \cM\big[M_1, \ldots, M_n \big] $.
\end{lemma}

Using the notation $ \otimes_{(r)} $ for the tensor product on $ \ta{r-1}{\vs} $, we note that by the above lemma $ \big(\ta{r}{\vs}, +, \otimes_{(r)}\big) $ is a multi-graded algebra over $ \vs $.
By recursively defining $ \Seq^r := \Seq(\Seq^{r-1}) $ with the convention that $ \Seq^0 = \{ \varnothing \} $. We may write down the multi-grading for $ \ta{r}{\vs} $ as follows
\begin{align}
\ta{1}{\vs}_{\bk} &:= \vs^{\otimes_{(1)}\bk}, \quad \ta{r}{\vs} = \bigoplus_{\bk \in \Seq^{r}} \ta{r}{\vs}_\bk, \\ 
\text{ where } \ta{r}{\vs}_{\bk} &:= \ta{r-1}{\vs}_{\bk_1} \otimes_{(r)} \cdots \otimes_{(r)} \ta{r-1}{\vs}_{\bk_l} \text{ for } \bk = \bk_1 \cdots \bk_l \in \Seq^{r}.
\end{align}
We also use the following recursive definition of the degree for a multi-index 
\begin{align}
\mathrm{deg}.\bk = \mathrm{deg}.\bk_1 + \cdots + \mathrm{deg}.\bk_l, \text{ for }
\bk = \bk_1 \cdots\bk_l \in \Seq^{r}, \quad \mathrm{deg}.\varnothing = 1.
\end{align} 
Which allows us to write down a grading for $ \ta{r}{\vs} $ as
\begin{align} \label{eq:grading}
\ta{r}{\vs} = \bigoplus_{k\geq0} \Big( \underset{\begin{subarray}{c}\bk \in \Seq^{r},\\ \mathrm{deg}.\bk = k\end{subarray}}\bigoplus \ta{r}{\vs}_\bk \Big).
\end{align}
See \cite[Section 3]{Ebrahimi15} for more on $ \ta{1}{\vs} $ and $ \ta{2}{\vs} $.

\begin{example}\quad
	\begin{itemize}
		\item 
		$ \ta{1}{\vs} $ is the standard tensor algebra over $ \vs $ and is graded over $ \Seq^1 \simeq \N $ by
		\begin{align}
		\ta{1}{\vs} = 1 + \bigoplus_{n\geq 1} \vs^{\otimes_{(1)} n}.
		\end{align}
		\item 
		$ \ta{2}{\vs} $ is graded over sequences in $ \N $ by
		\begin{align}
		\ta{2}{\vs} = 1 + \bigoplus_{n_1, \ldots, n_k \geq 1} \vs^{\otimes_{(1)} n_1}\otimes_{(2)}\cdots \otimes_{(2)}\vs^{\otimes_{(1)} n_k}.
		\end{align}
		\item 
		$ \ta{3}{\vs} $ is graded over matrices in $ \N $ by
		\begin{align}
		\ta{3}{\vs} = 1 + \underset{\begin{subarray}{c}n_1^1, \ldots, n_{k_1}^{1} \geq 1 \\[-7pt] \vdots \\[-4pt] n_1^{k_2}, \ldots, n_{k_1}^{k_2} \geq 1 \end{subarray}}\bigoplus 
		\big( \vs^{\otimes_{(1)} n_1^1}\otimes_{(2)}\cdots \otimes_{(2)}\vs^{\otimes_{(1)} n_{k_1}^1}\big) \otimes_{(3)} \cdots 
		\otimes_{(3)} \big( \vs^{\otimes_{(1)} n_1^{k_2}}\otimes_{(2)}\cdots \otimes_{(2)}\vs^{\otimes_{(1)} n_{k_1}^{k_2}}\big).
		\end{align}
	\end{itemize}
\end{example}

\begin{remark}
	For any ring $ R $ and $ R $ module $ M $ the above construction (disregarding the norm) yields a sequence of $ R $ algebras
	\begin{align}
	M \subseteq \ta{1}{M} \subseteq \ta{2}{M} \subseteq \cdots.
	\end{align}
	This sequence is characterized by the following universal property which follows from the universal property of the tensor algebra:
	\begin{center}
		For any $ R $ module $ N $, $ r\geq 0 $ and $ R $-module homomorphism $ \varphi : \ta{r}{M} \to N $ there exists a unique $ R $-algebra homomorphism $ \Phi : \ta{r+1}{M} \to N $ such that $ \varphi = \Phi \circ \iota $ where $ \iota $ is the inclusion map $ \ta{r}{M} \hookrightarrow \ta{r+1}{M} $.
	\end{center}
\end{remark}
\begin{example}
	For a concrete example of this, if $ \vs $ is some vector space over $ \R $ and $ X $ is a bounded random variable on $ \vs $, then its associated moment map is the linear map 
	\begin{align}
	\mu_X : \ta{1}{\vs} \to \R, \quad 
	\mu_X(e_{i_1}\cdots e_{i_k}) = \EE(X_{i_1}\cdots X_{i_k})
	\end{align}
	which induces the algebra homomorphism $ \mu^\star_X : \ta{2}{\vs} \to \R $. The reader familiar with \emph{cumulants} might note that the cumulants of $ X $ can then be described as linear functions of $ \mu^\star_X $. For example
	\begin{align}
	&\kappa(X_1,X_2,X_3) 
	= \EE(X_1X_2X_3) - \EE(X_1)\EE(X_2X_3) - \EE(X_2)\EE(X_1X_3) - \EE(X_3)\EE(X_1X_2) + 2\EE(X_1)\EE(X_2)\EE(X_3) \\
	&= \mu^\star_X(e_1e_2e_3) - \mu^\star_X(e_1\otimes_{(2)}e_2e_3) - \mu^\star_X(e_2\otimes_{(2)}e_1e_3) - \mu^\star_X(e_3\otimes_{(2)}e_1e_2) +2\mu^\star_X(e_1\otimes_{(2)}e_2\otimes_{(2)}e_3)
	\end{align}
\end{example}

\subsection{The multi-grading of $ \ata{r}{V} $}
Recall that the signature map, Definition~\ref{def:sig1}, takes paths on a vector space $ V $ as input and maps them into $ \Ta{1}{V} $ which is isomorphic to the completion of $ \ta{1}{V\oplus \R} $, so that in order to represent the signature of a path in $ V $ it is enough to represent elements of $  \ta{1}{V} $ for arbitrary finite dimensional $ V $.

In the rank two case we are not so lucky however, as $ \ata{2}{V} = \ata{1}{\ata{1}{V}} \simeq \ta{1}{\ta{1}{V\oplus \R}\oplus \R} $ which is not isomorphic to a rank 2 tensor algebra over any finite dimensional space. 
By definition, $ \ata{2}{V} $ is the free algebra generated by $ \ata{1}{V} $ and one indeterminate, so by Lemma \ref{lem:MG} it is multi-graded by $ \cM(\cM^1, \cM^0) $. We may write:
\begin{align}
\ata{2}{V} = \bigoplus_{m_1, \ldots, m_n \in \cM^1, \cM^0} \ata{2}{V}_{m_1} \otimes_{(2)} \cdots \otimes_{(2)} \ata{2}{V}_{m_n}
\end{align}
where $ \ata{2}{V}_{m} = \ta{1}{V}_{m} $ if $ m \in \cM^1 $ and $ \ata{2}{V}_{m} \simeq \R $ if $ m \in \cM^0 $. This multi-grading also allows us to write down a grading for $ \ata{2}{V} $ like in Equation \eqref{eq:grading} where the degree is defined in the natural way compatible with the degrees on $ \cM^1, \cM^0 $.

In the general case, $ \ata{r}{V} $ is the free algebra generated by $ \ata{r-1}{V} $ and one indeterminate, so its multi-grading may be recursively defined similarly. 

\subsection{Dimensions of the truncated spaces}
It is well known that if $ V $ is a $ d $-dimensional space, then $ V^{\otimes k} $ has dimension $ d^k $. Hence we can write (Recalling Equation \eqref{eq:grading}):
\begin{align}
\ta{1}{V} = \bigoplus_{k\geq 0} \ta{1}{V}_k, \quad  \dim. \ta{1}{V}_k = d^k.
\end{align}
In the case of $ \ta{2}{V} $ we may define $ \ta{2}{V}_k := \underset{\begin{subarray}{c}\bk \in \Seq^{2},\\ \mathrm{deg}.\bk = k\end{subarray}}\bigoplus \ta{2}{\vs}_\bk $ and write
\begin{align}
\ta{2}{V} = \bigoplus_{k\geq 0} \ta{2}{V}_k, \quad \dim. \ta{2}{V}_k = \frac12 (2d)^k.
\end{align}
To see why $ \dim. \ta{2}{V}_k = \frac12 (2d)^k $, note that since, as a vector space, $ \ta{2}{\vs}_\bk $ is isomorphic to $ \ta{1}{\vs}_k $ and hence for any $ \bk $ with $ \deg.\bk = k $, it also has dimension $ d^k $, so in order to determine the dimension of $ \ta{2}{V}_k $ it is enough to count $ \# \{ \bk \in \Seq^{2}: \mathrm{deg}.\bk = k \} = 2^{k-1} $.

In the case of $ \ata{1}{V} $ it is easily seen that $ \dim. \ata{1}{V}_k = (d+1)^k $, hence we may write 
\begin{align}
\ata{1}{V} = \bigoplus_{k\geq 0} \ata{1}{V}_k, \quad  \dim. \ata{1}{V}_k = (d+1)^k.
\end{align}
The case $ \ata{2}{V} $ is slightly more complicated, but can be characterised by a simple linear recursion.
\begin{prop}
	Let $ V $ be a $ d $-dimensional vector space, then 
	\begin{align}
	\ata{2}{V} = \bigoplus_{k\geq 0} \ata{2}{V}_k, \quad  \mathrm{dim}. \ata{2}{V}_k := A_{d+1}(k),
	\end{align}
	where $ A_d $ satisfies the recursion
	\begin{align}
	A_d(k) = (2d+1)A_d(k-1) - dA_d(k-2), \quad 
	A_d(0) = 1,\,\, A_d(1) = d+1
	\end{align}
\end{prop}
\begin{proof}
	Note that if $ k \in \cM^1 $, then $ \ata{2}{V}_k = \ata{1}{V}_k \simeq (V\oplus\R)^{\otimes k} $ and if $ k \in \cM^1 $, then $ \ata{2}{V}_k \simeq \R $, putting this together we get for $ \bk = k_1 \cdots k_l \in \cM(\cM^1, \cM^0) $
	\begin{align}
	\dim \ata{2}{V}_\bk = \prod_{k_i \in \cM^1} \dim \ta{2}{V}_{k_i} = (d+1)^{\sum_{k_i \in \cM^1} k_i}.
	\end{align}
	Assume that $ \bk = k_1 \cdots k_n $ and $ k_{i_1}, \ldots, k_{i_l} \in \cM^1 $, then since $ \deg k_i = 1 $ for any $ i \not= i_1, \ldots, i_l $ it must be true that $ \deg. \bk  = \sum_{k_i \in \cM^1} k_i + n - l $. By setting $ \ata{2}{V}_k := \underset{\begin{subarray}{c}\bk \in \cM(\cM^1,\cM^0),\\ \mathrm{deg}.\bk = k\end{subarray}}\bigoplus \ata{2}{\vs}_\bk $ we see that
	\begin{align}
	\dim \ata{2}{V}_k = \sum_{n=0}^k\sum_{l=0}^n \# \{ \bk = k_1 \cdots k_n \in \cM(\cM^1,\cM^0) : \deg. \bk = k,  k_{i_1}, \ldots, k_{i_{l}} \in \cM^1 \} (d+1)^{l + k - n}.
	\end{align}
	We note that for $ n,l $ fixed we have
	\begin{align}
	&\# \{ \bk = k_1 \cdots k_n \in \cM(\cM^1,\cM^0) : \deg. \bk = k, k_{i_1}, \ldots, k_{i_{l}} \in \cM^1 \} \\
	&= {n \choose l} \# \{ \bk = k_1 \cdots k_l \in \cM^2 : \deg. \bk = k+l-n \} \\
	&= {n \choose l}{k+l-n-1 \choose k-n}.
	\end{align}
	Hence
	\begin{align}
	\dim \ata{2}{V}_k = 1 + \sum_{n=1}^k\sum_{l=1}^n {n \choose l}{ k+l-n-1 \choose k-n}(d+1)^{l + k - n}.
	\end{align}
	By summing over the diagonal this can be rewritten as 
	\begin{align}
	\dim \ata{2}{V}_k &= 1 + \sum_{n=1}^k (d+1)^n \sum_{m=k-n+1}^k {m \choose m-k+n}{ n-1 \choose k-m}\\ &= 1 + \sum_{n=1}^k (d+1)^n (k + 1 - n){}_2F_1(1-n, k+2-n, 2, -1),
	\end{align}
	where $ {}_2F_1 $ is the Gaussian Hypergeometric function. It follows that for $ k \geq 2 $
	\begin{align}
	&\dim \ata{2}{V}_k - (2d+3)\dim \ata{2}{V}_{k-1} + (d+1)\dim \ata{2}{V}_{k-2} \\ 
	= \sum_{n=2}^{k-1} &(d+1)^n \Big[ (k + 1 - n){}_2F_1(1-n, k+2-n, 2, -1) - 2(k + 1 - n){}_2F_1(2-n, k+2-n, 2, -1) \\
	+ &(k - n){}_2F_1(2-n, k+1-n, 2, -1) - (k - n){}_2F_1(1-n, k+1-n, 2, -1) \Big] \\
	+ &(d+1)^k \big[ {}_2F_1(1-k, 2, 2, -1) - 2{}_2F_1(2-k, 2, 2, -1) \big] \\
	+ &(d+1) \big[ k{}_2F_1(0, k+1, 2, -1) - (k-1){}_2F_1(0, k, 2, -1) - 1 \big].
	\end{align}
	Because of the two facts 
	\begin{align}
	{}_2F_1(-k,2,2,-1) = 2^k, \quad 
	{}_2F_1(0,k,2,-1) = 1, 
	\end{align}
	all that remains is to show that for $ a > 0 $, $ F(-a,b) := {}_2F_1(-a,b,2,-1) $ satisfies the recursion 
	\begin{align}
	(b+1)F(-a,b+2) - 2(b+1)F(1-a,b+2) + bF(1-a,b+1) - bF(-a,b+1) = 0.
	\end{align}
	To see this, note that by expanding $ {}_2F_1(-a,b,2,-1) $ in its hypergeometric series we may write
	\begin{align}
	{}_2F_1(-a,b,2,-1) = \sum_{k = -\infty}^{\infty} f(a,b,k), \quad f(a,b,k) = {a \choose k} \frac {(b)_k}{(k+1)!} 1_{ \{ 0 \leq k \leq a \} },
	\end{align}
	where $ (b)_k $ is the rising Pochhammer symbol. It is straightforward to verify that $ f $ satisfies
	\begin{align}
	f(a+1,b,k) &= \frac{a+1}{a+1-k}f(a,b,k), \quad
	f(a,b+1,k) = \frac{b+k}{b}f(a,b,k), \\
	f(a,b,k+1) &= \frac{(a-k)(b+k)}{(k+1)(k+2)}f(a,b,k).
	\end{align}
	By iterating these three relations one can show that
	\begin{align}
	&(b+1)f(a,b+2,k) - 2(b+1)f(a-1,b+2,k) + bf(a-1,b+1,k) - bf(a,b+1,k) \\
	= &\frac{k(k+1)}a f(a,b+1,k) - \frac{(k+1)(k+2)}a f(a,b+1,k+1),
	\end{align}
	and the claimed recursion follows by summing over $ k $ since 
	\begin{align}
	&(b+1)F(-a,b+2) - 2(b+1)F(1-a,b+2) + bF(1-a,b+1) - bF(-a,b+1) \\
	&= \sum_{k = -\infty}^{\infty} (b+1)f(a,b+2,k) - 2(b+1)f(a-1,b+2,k) + bf(a-1,b+1,k) - bf(a,b+1,k) \\
	&= \sum_{k = -\infty}^{\infty} \frac{k(k+1)}a f(a,b+1,k) - \sum_{k = -\infty}^{\infty} \frac{(k+1)(k+2)}a f(a,b+1,k+1) = 0.
	\end{align}
\end{proof}

\begin{remark}
	The sequences $ A_0(k), A_1(k) $ are listed as $ A001519 $ and $ A052984 $ respectively on OEIS.
\end{remark}

\subsection{Higher rank tensor algebras on Banach spaces}

\begin{definition} \label{def:norm}
	We make $ \ata{r}{\vs} $ into a normed space with the norm defined inductively as 
	\begin{align}
	\lVert t \rVert_r = \sum_{m\geq 0} \lVert \pi_m t \rVert_{\ata{r-1}{\vs}^{\otimes m}}
	\end{align}
	where $ \pi_m : \ata{r}{\vs} \to \underset{\mathrm{deg}.\bk = m}\bigoplus \ata{r-1}{\vs}_\bk $ denotes the projection map onto components of degree $ k $.
	Define $ \Ta{r}{\Ban} $ to be the completion of $ \ata{r}{\Ban} $ under the norm $ \lVert \cdot \rVert_r $.
\end{definition}

\begin{remark} \label{rmk:embedding}
	Since the embedding $ \ata{r}{\Ban} \hookrightarrow \ata{r+1}{\Ban} $ is an isometric isomorphism onto its image, the same is true for the embedding $ \Ta{r}{\Ban} \hookrightarrow \Ta{r+1}{\Ban} $.
\end{remark}

\begin{remark}
	Note that $ \sigr{r} $ indeed takes values in $ \Ta{r}{E} $, since by the above Remark \ref{rmk:embedding} it is enough to show that $ \sigr{1} $ takes values in $ \Ta{1}{E} $ which follows from multiplication and addition being continuous and the exponential series being absolutely convergent.
\end{remark}

By unravelling Definition \ref{def:norm} we may write for $ t \in \ata{r}{V} $
\begin{align}
\lVert t \rVert_r = \sum_{\bk \in \cM^r} \lVert \pi_\bk t \rVert
\end{align}
where $ \pi_\bk : \ata{r}{V} \to \ata{r}{V}_\bk $ is projection onto $ \ata{r}{V}_\bk $ which is topologically isomorphic to a tensor copy of $ V $, hence it has a well defined norm by the assumption that $ V $ has an admissible norm. Finally, we note that if $\Ban$ is a Hilbert space, then $\Ta{r}{\Ban}$ also possesses a Hilbert space structure.
\begin{definition}\label{def: Hilbert tensor algebra}
	For a Hilbert space $ \vs $ we equip $ \ta{r}{\vs} $ with the recursively defined inner product 
	\begin{align}
	\langle t, s \rangle_r = \sum_{m\geq 0} \langle \pi_m t, \pi_m s \rangle_{\ta{r-1}{\vs}^{\otimes m}}
	\end{align}
	and we denote by $\tilde \cH^r(\vs) $ and  $ \cH^r(\vs) $ the respective completions of $ \ta{r}{\vs} $ and $ \ata{r}{\Ban} $ with this inner product.
\end{definition}
\subsection{Tensor normalization estimates} \label{app:normalization}
Recall that the scaling of an element $v \in \vs$ by $\lambda \in \R$, $\lambda \mapsto \lambda v$, extends naturally to a dilation map on $\prod_{m\ge 0}\Ban^{\otimes m}$:
$$
\delta_{\lambda}: \mathbf{t} \mapsto (\mathbf{t}^0, \lambda \mathbf{t}^1, \lambda^2 \mathbf{t}^2, \ldots).
$$
\begin{definition}\label{def: tensor normalization}
	A tensor normalization is a continuous injective map of the form
	$$
	\Lambda: \Ta{1}{\Ban} \rightarrow \{\mathbf{t} \in \Ta{r}{\Ban}: \|\mathbf{t}\| \le 1 \}, \quad \mathbf{t} \mapsto \delta_{\lambda(\mathbf{t})}\mathbf{t}, 
	$$
	where $\lambda: \Ta{r}{\vs} \rightarrow (0,\infty)$ is a function.
\end{definition}
It is possible to show that there always exists a tensor normalization, see \cite[Proposition A.2 and Corollary A.3]{Oberhauser18}.
\begin{theorem}
	For any Banach space $\vs$ and any admissible norm on $(\vs^{\otimes m})_{m \ge 1}$, there exists a tensor normalization map $\Lambda$.
\end{theorem}
For any (discrete time) path $x \in (I \rightarrow \Ban)$, $\Sig{x}$ takes values in $\Ta{1}{\vs}$, see \cite[Theorem 3.12]{Lyons02}. In particular, for a given tensor normalization $\Lambda$,  $\Lambda \circ \Sig{x}$ takes values in the unit ball of the Banach space $\Ta{1}{\vs}$, and therefore for any $\mu \in \Meas{\idx \to \vs}$, the Bochner integral $\overline {\Lambda \circ \Sigo}(x) :=\int_{x \in \Ban^{\idx}} \Lambda \circ \Sig{x} \mu(dx)$ is well--defined. Then we may iteratively define $ \Lambda \Sigo^r := \Lambda \circ \Sigo \circ \Lambda \Sigo^{r-1} $

\begin{definition}\label{def: higher rank normalized signature}
	We call $ \sigrno{r} := \Lambda \Sigo^{r} $ the robust (or, normalized) signature map of rank $ r $ and $\esigrno{r} := \EE\sigrno{r}$ is called the robust expected signature map of rank $ r $.
\end{definition}

The proof of the next proposition can be found in \cite[Corollary 5.7]{Oberhauser18}.
\begin{proposition}\label{prop: expected signature is injective}
	Let $ \vs $ be a separable Banach space. Then $\esigrno{1}: \Meas{\idx \to \vs} \rightarrow \Ta{1}{\vs}$ is injective. 
\end{proposition}

For our concrete purpose in Sect. \ref{sec:tightness}, we introduce the following robust signature, which is slightly different from the one we defined above as it is not of the form $\Lambda \circ \Sigo$.
\begin{proposition}\label{prop: a feasible tensor normalization}
Let $V$ be a separable Banach space. Let $\Phi : (\idx \to \vs) \to \Ta{1}{\vs}$ be the map such that for $x \in (\idx \to V)$,
$$
\Phi(x) = \delta_{\exp(-\|\Sig{x}\|-\|x\|_{\infty})}\Sig{x}.
$$
Then $\Phi$ is bounded continuous and injective.
\end{proposition}
\begin{proof}
Let $\mathbbm{1}$ denote the neutral element $(1,0,\ldots)$ in $\Ta{1}{\vs}$ with respect to the tensor product.\\
The boundedness of $\Phi$ is clear, because for any $a \in [0,1]$ it holds that $\|\delta_a \Sig{x} -\mathbbm{1}\| \le a\|\Sig{x}\|$,  inserting $a = \exp(-\|\Sig{x}\|-\|x\|_{\infty})$ we indeed get a uniform bound for $\Phi$. To show the continuity of $\Phi$, note that for $x^n$ converges to $x$ in $\idx \to \vs$,  we have
\begin{align*}
\|\Phi(x^n) - \Phi(x)\| \le &\|\delta_{\exp(-\|\Sig{x^n}\|-\|x^n\|_{\infty})}\Sig{x^n} - \delta_{\exp(-\|\Sig{x^n}\|-\|x^n\|_{\infty})}\Sig{x}\| \\
&+ \|\delta_{\exp(-\|\Sig{x^n}\|-\|x^n\|_{\infty})}\Sig{x} - \delta_{\exp(-\|\Sig{x}\|-\|x\|_{\infty})}\Sig{x}\|.
\end{align*}
The first term on the right hand side converges to $0$ as it is bounded by $\|\Sig{x^n} - \Sig{x}\|$ which vanishes as $n \to \infty$ by the continuity of $\Sigo$; the second term on the right hand side also tends to $0$ by dominated convergence as $\Sig{x}$ has a factorial decay in its tail (cf. \cite[Theorem 3.1.2]{Lyons02}). For the injectivity of $\Phi$, let us assume that $\Phi(x) = \Phi(y)$ for $x,y \in \idx \to \vs$. Using the relation that $\delta_a \delta_b \Sig{x} = \delta_{ab} \Sig{x}$ for all $a,b \in \R$ it implies that
$$
\delta_c \Sig{x} = \Sig{y},
$$
where $c = \exp(\|\Sig{y}\| + \|y\|_\infty - \|\Sig{x}\| - \|x\|_\infty)$. Then by \cite[Exercise 7.55]{Fritz10} we have $\delta_c(\Sig{x}) = \Sig{cx} = \Sig{y}$. However, keeping in mind that we included time component into the definition of $\Sigo$, it holds that the projection of $\Sig{cx}$ to the first level $\R \oplus \vs$ is equal to $(cT, cx_T)$ while the counterpart for $\Sig{y}$ is $(T,y_T)$. Therefore we must have $cT = T$, i.e., $c = 1$. Consequently it follows that $\Sig{x} = \Sig{y}$ and also $x=y$ by the injecvitity of $\Sigo$ (see Theorem \ref{cor:injective}).
\end{proof}
\begin{remark}
Note that the injectivity of $\Phi$ depends crucially on the fact that we include time component into the definition of signature map, and it may not be true if one uses signature without time extension. In the latter case one has to apply the tensor normalization introduced in \cite{Oberhauser18}. Also note that we include $\|x\|_\infty$ into the dilation for a special technical reason, see discussion in the next section.
\end{remark}
\section{Feature Maps, MMDs and Weak Convergence}\label{app: signature, feature map and MMD}
In this Section we provide the necessary background for the robust signature map $\sigrno{r}$ that we use to deal with non-compactness, see Section~\ref{sec:tensor normalization}. 
Central to our argument is to exploit a duality between functions and measures via a ``universal feature map''.
In the non-compact case this duality can be subtle to handle, see \cite{Simon20} for an overview. 
\subsection{Universality and Characteristicness}
\begin{definition}\label{def: universal feature map}
Let $\mathcal{X}$ be a topological space and $E$ be a topological vector space. We call any map $\Phi; \mathcal{X} \rightarrow E$ a feature map. Moreover, for a given topological vector space $\cF \subset \R^{\mathcal{X}}$, we say that a feature map $\Phi$ is 
\begin{enumerate}
	\item universal to $\cF$, if the map
	$$
	\iota: E^\prime \rightarrow \R^{\mathcal{X}}, \quad \ell \mapsto \langle \ell, \Phi(\cdot)\rangle
	$$
	has a dense image in $\cF$, where $E^\prime$ denotes the topological dual of $E$.
	\item characteristic to a subset $\cP \subset \cF^\prime$ if the map
	$$
	\kappa: \cP \rightarrow (E^\prime)^*, \quad D \mapsto [\ell \mapsto D(\langle \ell, \Phi(\cdot) \rangle]
	$$
	is injective, where $(E^\prime)^*$ denotes the algebraic dual of $E^\prime$.
\end{enumerate}
\end{definition}
The following duality is a direct consequence of the Hahn--Banach Theorem, see e.g.~\cite[Theorem 2.3]{Oberhauser18}.
\begin{theorem}\label{thm: duality of universality and characteristicness}
If $\cF$ is a locally convex space, then a feature map $\Phi$ is universal to $\cF$ if and only if $\Phi$ is characteristic to $\cF^\prime$.
\end{theorem}

\subsection{Robust Signature Features and their Topology}
Put in our context, the feature map is the (robust) signature map $\sigrno{r}$.
\begin{theorem}\label{thm: properties of signature and normalzied signature}
  Define $\Delta(v) := 1 \otimes v + v \otimes 1$  for $v \in \Ban$.
  Then $(\Ta{1}{\Ban}, \otimes, \Delta)$ is a Hopf algebra and the co-domain of
  both the signature map $\Sigo$ and the robust signature map $\sigrno{1}$ is the set of group-like elements 
  \[
    G \coloneqq \{g \in \Ta{1}{\Ban}: \Delta g = g \otimes g\}\subset \Ta{1}{\Ban},
    \]
    that is
    \begin{align}\label{eq: codomain grouplike}
      \Sigo, \sigrno{1} \,:\, (\idx \rightarrow E) \rightarrow G \subset \Ta{1}{\Ban}.
    \end{align}
Moreover, both these maps are continuous and injective.
\end{theorem}
\begin{proof}
  This is classical for $\Sigo$ and follows for $\sigrno{1}$ by integration by parts, see~\cite[Sect. 5.1]{Oberhauser18}. 
\end{proof}
The following proposition is crucial for this present paper.
\begin{proposition}\label{prop: central proposition}
  Assume that $\mathcal{X}$ is metrizable.
 Then for any continuous injective mapping $\varphi: \mathcal{X} \rightarrow \Ban$, where $\Ban$ is a Banach space, the map \[\Phi := \Sigo^n_1 \circ \varphi: \mathcal{X}^I \rightarrow \Ta{1}{\Ban}\] is universal to $C_b(\mathcal{X}^I, \R)$ and characteristic to $C_b(\mathcal{X}^I, \R)^\prime$. In particular, two finite regular Borel measures $\mu$ and $\nu$ on $\mathcal{X}^I$ are equal if and only if $\int \Phi d\mu = \int \Phi d\nu$.
\end{proposition}
\begin{proof}
Let $x \in (\idx \rightarrow \mathcal{X})$ be a (discrete time) path taking values in $\mathcal{X}$. Then $\varphi \circ x$ is a (discrete time) path taking values in $\Ban$. Thanks to Theorem \ref{thm: properties of signature and normalzied signature} the map $\Phi = \Sigo^n_1 \circ \varphi$ is continuous and injective, and takes values in $G \subset \Ta{1}{\Ban}$. Define $L := \bigoplus_{m \ge 0}(\Ban^{\otimes m})^\star$, which we identify with a dense subspace of $\Ta{1}{\Ban}^\star$ via $\ell(\mathbf{t}) = \sum_{m \ge 0}\langle \ell^m, \mathbf{t}^m \rangle$, and define $\tilde \cF := \{ \ell \circ \Phi: \ell \in L \}$. Clearly, $\tilde \cF \subset \cF = C_b(\mathcal{X}^I, \R)^\prime$, and the injectivity of $\Phi$ implies that $\tilde \cF$ separates the points in $\mathcal{X}^\idx$. Furthermore, the algebraic condition on $G$ implies that $\tilde \cF$ is closed under multiplication (when $\Ban = \R^d$, this is equivalent to the shuffle product equation in \cite[(2.6)]{Lyons07}). This implies that $\tilde \cF$ satisfies all conditions in \cite[Theorem 2.6, (2)]{Oberhauser18} and is therefore dense in $C_b(\mathcal{X}^I, \R)$ with respect to the strict topology by \cite[Theorem 2.6]{Oberhauser18}. This means that $\tilde \cF$ is universal to $C_b(\mathcal{X}^I, \R)$ and characteristic to $C_b(\mathcal{X}^I, \R)^\prime$ by \cite[Theorem 2.3]{Oberhauser18}. The last assertion then follows immediately. 
\end{proof}
\subsection{Kernelized Maximum Mean Discrepancies}
Following \cite{Oberhauser18} we now use $\sigrno{1}$ to define a kernel and show that the associated Maximum Mean Discrepancy (MMD) metrizes weak convergence when the state space $\Ban =\R^d$. 

\begin{proposition}\label{prop: normalized signature satisfies nice conditions}
  Let $V = \R^d$ 
  Define
  \begin{align}
 \mathbf{k}:(\idx \to \Ban) \times (\idx \to \Ban) \to \R,\quad    \mathbf{k}(x,y) := \langle \sigrn{1}{x}, \sigrn{1}{y} \rangle - 1.
  \end{align}
  Then \begin{enumerate}
	\item $\mathbf{k}$ is a continuous, bounded, positive definite function.
	\item $\mathbf{k}$ is characteristic to $C_b(\idx \to \vs)^\prime$.
	\item The the reproducing kernel Hilber space (RKHS) generated by $\mathbf{k}$, $\cH_{\mathbf{k}}$, is a subset the space of all continuous functions on $\idx \to \vs$ vanishing at infinity,  \[\cH_{\mathbf{k}} \subset C_0(\idx \to \vs).\]
\end{enumerate} 
\end{proposition}
\begin{proof}
  The first statement follows immediately from Proposition~\ref{prop: a feasible tensor normalization} as $\sigrno{r}$ is a bounded continuous mapping with values in $\cH^1(\vs)$.
  To see characteristicness, note that $\mathbf{k}(x,y) = \langle \sigrn{1}{x}-\mathbbm{1}, \sigrn{1}{y}-\mathbbm{1}\rangle$ for $\mathbbm{1} = (1,0,\ldots) \in \cH^1(\vs)$ the unit element.
  By~\cite[Proposition 7.3]{Oberhauser18} it then remains to show that ${ \sigrnhat{1}{x}} \coloneqq {\sigrn{1}{x}} -\mathbbm{1}$ is characteristic to $C_b(\idx \to \vs)^\prime$.
  However, this property is inherited from the corresponding property of $\sigrno{1}$, Proposition \ref{prop: central proposition}.
  To be precise, the construction of $\sigrno{1}$ ensures that $\sigrn{1}{x}$ is group--like.
  Consequently $\{\langle \ell, \sigrn{1}{x} \rangle: \ell \in \ata{1}{\Ban}\}$ forms an algebra, which in turn implies that $\{\langle \ell, { \sigrnhat{1}{x}} \rangle: \ell \in \ata{1}{\Ban}\}$ is also an algebra as ${ \sigrnhat{1}{x}}$ coincides with $\sigrnhat{1}{x}$ on $\Pi_{m=1}^\infty (\R \oplus \vs)^{\otimes m}$ and the projection of ${ \sigrnhat{1}{x}}$ to $(\R \oplus \vs)^{\otimes 0} \sim \R$ equals $0$. Furthermore, the boundedness and continuity of $\sigrno{1}$ ensures that $\{\langle \ell, \sigrnhat{1}{x} \rangle: \ell \in \ata{1}{\Ban}\} \subset C_b(\idx \to \vs)$; the injectivity guarantees that $\{\langle \ell, \sigrnhat{1}{x} \rangle: \ell \in \ata{1}{\Ban}\}$ separates points; finally, since each $\sigrnhat{1}{x}$ contains time component $\exp(-\|\Sig{x}\| - \|x\|_\infty) T \neq 0$ as we are using time extended signature, the set $\{\langle \ell, \sigrnhat{1}{\cdot} \rangle: \ell \in \ata{1}{\Ban}\}$ still contains constant functions. Hence, we can use a Stone--Weierstrass type argument as in Proposition \ref{prop: central proposition}, see also \cite[Theorem 2.6]{Oberhauser18}, to deduce that $\sigrnhat{1}{\cdot}$ is characteristic to $C_b(\idx \to \vs)^\prime$.

Finally, note that for $x \in \idx \to \vs$, one has $\lim_{\|y\|_\infty \to \infty} |\mathbf{k}(x,y)| = 0$, because
\begin{align*}
|\mathbf{k}(x,y)| = |\langle \sigrnhat{1}{x}, \sigrnhat{1}{y}\rangle| &\le \|\sigrnhat{1}{x} \| \|\sigrnhat{1}{y}\| \\
&= \|\sigrnhat{1}{x} \|\Big\|\delta_{\exp(-\|\Sig{y}\|-\|y\|_{\infty})}\Sig{y} - \mathbbm{1}\Big\|\\
&\le \|\sigrnhat{1}{x} \|\exp(-\|\Sig{y}\|-\|y\|_{\infty})\|\Sig{y}\|\\
&\to 0
\end{align*}
as $\|y\|_\infty \to \infty$.
Hence, in view of~\cite[Lemma 4.1]{Simon20} one can conclude that $\cH_{\mathbf{k}} \subset C_0(\idx \to \vs)$.
\end{proof}
We now conclude by \cite[Lemma 2.1]{Simon20} 

\begin{corollary}\label{cor: MMD characterizes weak convergence in appendix}
Let $\vs= \R^d$. Then 
$$
d_{\mathbf{k}}(\mu, \nu) = \Big\|\int \sigrn{r}{x} \mu(dx) - \int \sigrn{r}{x} \nu(dx)\Big\|  = \Big\|\esigrno{1}(\mu) - \esigrno{1}(\nu) \Big\|
$$
characterizes weak convergence.
\end{corollary}
\end{document}